\documentclass[12pt, a4paper, reqno]{amsart}
\usepackage{amsmath}
\usepackage{amsfonts}
\usepackage{amssymb}
\usepackage{color}
\usepackage{comment,graphicx}
\usepackage[T1]{fontenc}
\usepackage{url}
\usepackage{ae}

\def\phi{\varphi}
\def\rho{\varrho}
\def\epsilon{\varepsilon}
\numberwithin{equation}{section}
\theoremstyle{plain}
\newtheorem{thm}[equation]{Theorem}
\newtheorem{lem}[equation]{Lemma}

\newtheorem{cor}[equation]{Corollary}
\theoremstyle{defn}
\newtheorem{defn}[equation]{Definition}

\theoremstyle{remark}
\newtheorem{rem}[equation]{Remark}

\providecommand{\loc}{{\ensuremath{\mathrm{loc}}}}

\renewcommand{\leq}{\leqslant}
\renewcommand{\geq}{\geqslant}

\begin{document}
\author[D. Drihem]{Douadi Drihem \ }
\date{\today }
\title[Nemytzkij operators and  semi linear parabolic equations]{%
Composition operators on Herz-type Triebel-Lizorkin spaces with
application to semilinear parabolic equations}
\maketitle

\begin{abstract}
Let $G:\mathbb{R\rightarrow R}$ be a continuous function. In the first part
of this paper, we investigate sufficient conditions on $G$ such that 
\begin{equation*}
\{G(f):f\in \dot{K}_{p,q}^{\alpha }F_{\beta }^{s}\}\subset \dot{K}%
_{p,q}^{\alpha }F_{\beta }^{s}
\end{equation*}%
holds. Here $\dot{K}_{p,q}^{\alpha }F_{\beta }^{s}$ are Herz-type
Triebel-Lizorkin spaces. These spaces unify and generalize many classical
function spaces such as Lebesgue spaces of power weights, Sobolev and
Triebel-Lizorkin spaces of power weights. In the second part of this paper
we will study local and global Cauchy problems for the semilinear parabolic
equations%
\begin{equation*}
\partial _{t}u-\Delta u=G(u)
\end{equation*}%
with initial data in Herz-type Triebel-Lizorkin spaces. Our results cover
the results obtained with initial data in some know function spaces such us
\ fractional Sobolev spaces. Some limit cases are given.

\noindent MSC classification (2010): 46E35, 47H30, 35K45, 35K55.\newline
\end{abstract}

\section{Introduction}

Let $G:\mathbb{R\rightarrow R}$ be a function. In this paper we consider the
Cauchy problem for semilinear parabolic equations on $\mathbb{R}^{n}$ of the
following form:%
\begin{equation}
\frac{\partial u}{\partial t}(t,x)=\Delta u(t,x)+G(u(t,x)),\quad (t,x)\in 
\mathbb{(}0,\mathbb{\infty )}\times \mathbb{R}^{n}  \label{equation}
\end{equation}%
subject to the initial value condition%
\begin{equation*}
u(0,x)=u_{0}(x)\quad \text{on}\quad \mathbb{R}^{n}.
\end{equation*}%
The most classical examples of such equations are the semilinear heat
equations%
\begin{equation}
\frac{\partial u}{\partial t}(t,x)=\Delta u(t,x)+u|u|^{\mu -1},\quad
(t,x)\in \mathbb{(}0,\mathbb{\infty )}\times \mathbb{R}^{n},\mu >1,
\label{equation1}
\end{equation}%
the Burgers viscous equations%
\begin{equation*}
\frac{\partial u}{\partial t}(t,x)=\Delta u(t,x)+\partial _{x}(|u|^{\mu
}),\quad (t,x)\in \mathbb{(}0,\mathbb{\infty )}\times \mathbb{R}^{n},\mu >1
\end{equation*}%
and the Navier-Stokes equation%
\begin{equation*}
\frac{\partial u}{\partial t}(t,x)=\Delta u(t,x)+\mathcal{P}\nabla (u\otimes
u),\quad (t,x)\in \mathbb{(}0,\mathbb{\infty )}\times \mathbb{R}^{n},\mu >1,
\end{equation*}%
where $\mathcal{P}$ denotes the projector on the divergence free vector
field. Let us recall briefly some results on most known function spaces. For
Lebesgue space,  Weissler in \cite{W79} and \cite{We80} studied %
\eqref{equation1} with singular data in certain Lebesgue spaces $L^{p}$. In 
\cite{W79} he proved\ the local existence of \eqref{equation1} with
initially data in $L^{p_{c}}$ with $p_{c}=\frac{n(\mu -1)}{2}>1$ and the
solution belongs to $C([0,T),L^{p})$, and that $T$ can be taken as infinity
for sufficiently\ small data in $L^{p_{c}}$. Giga \cite{Gi86} proved that
the solution belongs to $L^{q}([0,T),L^{p})$ with $\frac{1}{q}=\frac{n}{2}(%
\frac{1}{p_{c}}-\frac{1}{p}),p,q>p_{c}$ and $q>\mu .$

Weissler \cite{W79} proved the local existence of \eqref{equation1} for
initial values in $L^{p}$ with $p>p_{c}$ and $p\geq \mu $. See \cite{Gi86}
for further results.

In case of $1<p<p_{c}$ there exist some non-negative initial data in $L^{p}$
for which there is no non-negative solution for any positive time $T>0$, see
e.g. \cite{BK96} and \cite{We80}.

Further results, for the well-posedness of the Cauchy problem of %
\eqref{equation1} can be found in \cite{CW19}, \cite{TW1}, \cite{Te02} and \cite{U21}.

In the framework of fractional Sobolev spaces, \cite{Ri98} established local
well-posedness of problem \eqref{equation} with some suitable assumptions on 
$G$ and obtained existence of global small solutions in $H_{p}^{\frac{n}{p}-%
	\frac{2}{\mu }}$. Miao and Zhang, \cite{MZ04} establish the local
well-posedness and small global well-posedness in Besov spaces $B_{p,2}^{s}$%
. Also, they establish the local well-posedness and small global
well-posedness of problem \eqref{equation} in the critical space $B_{p,2}^{%
	\frac{n}{p}}$.

In \cite{Dr-Herz-Heat} the author\ study the equation \eqref{equation} with%
\begin{equation}
|G(x)-G(y)|\leq |x-y|(|x|^{\mu -1}+|y|^{\mu -1}),\quad x,y\in \mathbb{R},\mu
>1,G(0)=0  \label{Ass-G}
\end{equation}%
and initial data in Herz spaces $\dot{K}_{p,q}^{\alpha }$. Herz spaces play
an important role in Harmonic Analysis. After they have been introduced in 
\cite{Herz68}, the theory of these spaces had a remarkable development in
part due to its usefulness in applications. For instance, they appear in the
characterization of multipliers on Hardy spaces \cite{BS85}, in the
summability of Fourier transforms \cite{FeichtingerWeisz08} and in
regularity theory for elliptic equations in divergence form \cite{Rag09}.
They unify and generalize the classical Lebesgue spaces of power weights.
More precisely, if $\alpha =0$ and $p=q$, then $\dot{K}_{p,p}^{0}$ coincides
with the Lebesgue spaces $L^{p}$ and 
\begin{equation*}
\dot{K}_{p,p}^{\alpha }=L^{p}(\mathbb{R}^{n},|\cdot |^{\alpha p}),\quad 
\text{(Lebesgue space equipped with power weight).}
\end{equation*}%

The aims of the present paper is to study the equation \eqref{equation} in
Herz-type Triebel-Lizorkin spaces $\dot{K}_{p,q}^{\alpha }F_{\beta }^{s}$.
These spaces unify and generalize the classical Lebesgue spaces of power
weights, fractional Sobolev spaces of power weights and Triebel-Lizorkin
spaces of power weights. We will assume that $G\ $belongs to $G\in Lip\mu $,
see Section 3 for the definition of the spaces $Lip\mu $.

We recall that the solution in the function space $\dot{K}_{p,q}^{\alpha }F_{\beta }^{s}$ of the integral equation

\begin{equation}
u(t,x)=e^{t\Delta }u_{0}+\int_{0}^{t}e^{(t-\tau )\Delta }G(u)(\tau ,x)d\tau 
\label{int-equa1}
\end{equation}
is usually defined as the mild solution of the  Cauchy problem \eqref{equation}. Under some assumption on $p,q\ \beta ,\alpha $ and $s$ we prove that for all initial
data $u_{0}$ in $\dot{K}_{p,q}^{\alpha }F_{\beta }^{s}$ with $s>\bar{s}=%
\frac{n}{p}+\alpha -\frac{2}{\mu -1}$, there exists a maximal solution $u$
to \eqref{int-equa1} in $C([0,T_{0}),\dot{K}_{p,q}^{\alpha }F_{\beta }^{s})$
with $T_{0}\geq C\big\|u_{0}\big\|_{\dot{K}_{p,q}^{\alpha }F_{\beta }^{s}}^{-%
	\frac{1}{\vartheta }}$. If $\theta <(s-\bar{s})(\mu -1)$, then we prove that 
\begin{equation}
u-e^{t\Delta }u_{0}\in C([0,T_{0}),\dot{K}_{p,q}^{\alpha }F_{\beta
}^{s+\theta }).  \label{Ribaud}
\end{equation}%
Now if $\theta =(s-\bar{s})(\mu -1),s>1$ with $G\in Lips_{0}$\ and 
\begin{equation*}
s_{0}=\frac{\frac{n}{p}+\alpha }{\frac{n}{p}+\alpha -s+1},
\end{equation*}%
then we have \eqref{Ribaud}, which was not treated in \cite{Ri98}. Our
results cover the corresponding results of \cite{Ri98}. Moreover, we present
the limit case 
\begin{equation*}
s=1+\frac{\mu -1}{\mu }\big(\frac{n}{p}+\alpha \big)
\end{equation*}%
and the case\ when $s>\frac{n}{p}+\alpha $. To study \eqref{equation} we
investigate sufficient conditions on $G$ such that 
\begin{equation*}
\{G(f):f\in \dot{K}_{p,q}^{\alpha }F_{\beta }^{s}\}\subset \dot{K}%
_{p,q}^{\alpha }F_{\beta }^{s}.
\end{equation*}%
In Sobolev space, \cite{MM} have presented the necessary and sufficient
conditions on $G$ such that 
\begin{equation*}
G(W_{p}^{1}(\mathbb{R}^{n}))\subset W_{p}^{1}(\mathbb{R}^{n}),
\end{equation*}%
except the case $p=n\geq 2$. A complete characterization of this problem in
Sobolev spaces\ has been given by Bourdaud in \cite{Bo91} and \cite{Bo10}.
The surprise result in Sobolev spaces is that under some assumptions there
is no non-trivial function $G$ which acts via left composition on such
spaces. More precisely, in 1978 Dahlberg \cite{Da79} proved that%
\begin{equation*}
G(f)\in W_{p}^{m}(\mathbb{R}^{n}),\quad f\in W_{p}^{m}(\mathbb{R}^{n}),\quad
1<p<\infty ,\quad 2\leq m<\frac{n}{p}
\end{equation*}%
implies $G(t)=ct$ for some $c\in \mathbb{R}$. In the framework of Sobolev
spaces with fractional order, $H^{s}(\mathbb{R}),0<s<1,s\neq 2$, Igari in 
\cite{Igari65} gave the necessary and sufficient conditions on $G$ such that 
$G(H^{s}(\mathbb{R}))\subset H^{s}(\mathbb{R})$. He observed the necessity
of local Lipschitz continuity for the first time. See \cite{J89} for the
Hardy-Sobolev space $F_{2}^{1,2}(\mathbb{R}^{n})$.

The extension of the above results to Besov and Triebel-Lizorkin spaces is
given by Bourduad in \cite{Bo93} and \cite{Bo931}, Runst in \cite{Ru86}, and
Sickel in \cite{Si97}, \cite{Si98} and \cite{Si98-1}. Further results
concerning the composition operators in Besov and Triebel-Lizorkin spaces
are given \cite{BK}, \cite{BCS06}, \cite{Bou-Cri08}, \cite{BMS10}, \cite%
{BM01} and \cite{RS96}.
Recently, Bourdaud and Moussai \cite{BM19} proved the continuity of
the composition operator in $W_{p}^{m}(\mathbb{R}^{n})\cap \dot{W}_{mp}^{1}(%
\mathbb{R}^{n})$ to itself, for every integer $m\geq 2$ and any $1\leq
p<\infty $ and in Sobolev spaces $W_{p}^{m}(\mathbb{R}^{n})$, with $m\geq 2$
and $1\leq p<\infty $.
The author in \cite{Dr21.1} and \cite{Dr21.3} gave the necessary
and sufficient conditions on $G$ such that 
\begin{equation*}
G(W_{p}^{m}(\mathbb{R}^{n},|\cdot |^{\alpha }))\subset W_{p}^{m}(\mathbb{R}%
^{n},|\cdot |^{\alpha }),\quad \text{(Sobolev space of power
	weight),}
\end{equation*}%
with some suitable assumptions on $m,p$ and $\alpha $. The extension of
Dahlberg result to Triebel-Lizorkin spaces of power weights $F_{p,q}^{s}(%
\mathbb{R}^{n},|\cdot |^{\alpha })$ is given in \cite{Dr21.2}.

\subsection{Notation and conventions}

Throughout this paper, we denote by $\mathbb{R}^{n}$ the $n$-dimensional
real Euclidean space, $\mathbb{N}$ the collection of all natural numbers and 
$\mathbb{N}_{0}=\mathbb{N}\cup \{0\}$. The letter $\mathbb{Z}$ stands for
the set of all integer numbers.\ The expression $f\lesssim g$ means that $%
f\leq c\,g$ for some independent constant $c$ (and non-negative functions $f$
and $g$), and $f\approx g$ means $f\lesssim g\lesssim f$. As usual for any $%
x\in \mathbb{R}$, $\left\lfloor x\right\rfloor $ stands for the largest
integer smaller than or equal to $x$.\vskip5pt

For $x\in \mathbb{R}^{n}$ and $r>0$ we denote by $B(x,r)$ the open ball in $%
\mathbb{R}^{n}$ with center $x$ and radius $r$. By supp$f$ we denote the
support of the function $f$, i.e., the closure of its non-zero set. If $%
E\subset {\mathbb{R}^{n}}$ is a measurable set, then $|E|$ stands for the
(Lebesgue) measure of $E$ and $\chi _{E}$ denotes its characteristic
function. For any $u>0$, we set $C(u)=\{x\in \mathbb{R}^{n}:\frac{u}{2}%
<\left\vert x\right\vert \leq u\}$. By $c$ we denote generic positive
constants, which may have different values at different occurrences. \vskip%
5pt

Given a measurable set $E\subset \mathbb{R}^{n}$ and $0<p\leq \infty $, we
denote by $L^{p}(E)$ the space of all functions $f:E\rightarrow \mathbb{C}$
equipped with the quasi-norm 
\begin{equation*}
\big\|f\big\|_{L^{p}(E)}=\Big(\int_{E}\left\vert f(x)\right\vert ^{p}dx\Big)%
^{1/p}<\infty
\end{equation*}%
with $0<p<\infty $ and%
\begin{equation*}
\big\|f\big\|_{L^{\infty }(E)}=\underset{x\in E}{\text{ess-sup}}\left\vert
f(x)\right\vert <\infty .
\end{equation*}%
If $E=\mathbb{R}^{n}$, then we put $L^{p}(\mathbb{R}^{n})=L^{p}$ and $\big\|f%
\big\|_{L^{p}(\mathbb{R}^{n})}=\big\|f\big\|_{p}.$

Let $w$ denote a positive, locally integrable function and $0<p<\infty $.
Then the weighted Lebesgue space $L^{p}(\mathbb{R}^{n},w)$ contains all
measurable functions $f$ such that 
\begin{equation*}
\big\|f\big\|_{L^{p}(\mathbb{R}^{n},w)}=\Big(\int_{\mathbb{R}^{n}}\left\vert
f(x)\right\vert ^{p}w(x)dx\Big)^{1/p}<\infty .
\end{equation*}%
If $1\leq p\leq \infty $ and $\frac{1}{p}+\frac{1}{p^{\prime }}=1$, then $%
p^{\prime }$ is called the conjugate exponent of $p$.

By $\mathcal{S}(\mathbb{R}^{n})$ we denote the Schwartz space of all
complex-valued, infinitely differentiable and rapidly decreasing functions
on $\mathbb{R}^{n}$ and by $\mathcal{S}^{\prime }(\mathbb{R}^{n})$ the dual
space of all tempered distributions on $\mathbb{R}^{n}$. We define the
Fourier transform of a function $f\in \mathcal{S}(\mathbb{R}^{n})$ by%
\begin{equation*}
\mathcal{F}(f)(\xi )=\left( 2\pi \right) ^{-n/2}\int_{\mathbb{R}%
	^{n}}e^{-ix\cdot \xi }f(x)dx.
\end{equation*}%
Its inverse is denoted by $\mathcal{F}^{-1}f$. Both $\mathcal{F}$ and $%
\mathcal{F}^{-1}$ are extended to the dual Schwartz space $\mathcal{S}%
^{\prime }(\mathbb{R}^{n})$ in the usual way.

For $v\in \mathbb{Z}$ and $m=(m_{1},...,m_{n})\in \mathbb{Z}^{n}$, let $%
Q_{v,m}$ be the dyadic cube in $\mathbb{R}^{n}$, $Q_{v,m}=%
\{(x_{1},...,x_{n}):m_{i}\leq 2^{v}x_{i}<m_{i}+1,i=1,2,...,n\}$. Also, we
set $\chi _{j,m}=\chi _{Q_{j,m}},j\in \mathbb{Z},m\in \mathbb{Z}^{n}.$

Recall that $\eta _{R,m}(x)=R^{n}(1+R\left\vert x\right\vert )^{-m}$, for
any $x\in \mathbb{R}^{n}$ and $m,R>0$. Note that $\eta _{R,m}\in L^{1}(%
\mathbb{R}^{n})$ when $m>n$ and that $\left\Vert \eta _{R,m}\right\Vert
_{1}=c_{m}$ is independent of $R$, where this type of function was
introduced in \cite{DHR} and \cite{HN07}.

\section{Function spaces}

In this section we\ present the Fourier analytical definition of Herz-type
Triebel-Lizorkin spaces and we present their basic properties such us Sobolev
embeddings. We start by recalling the definition and some  properties
of Herz spaces. For convenience, we set 
\begin{equation*}
B_{k}=B(0,2^{k}),\quad \bar{B}_{k}=\{x\in {\mathbb{R}^{n}:|x|\leq }%
2^{k}\},\quad k\in \mathbb{Z}
\end{equation*}%
and 
\begin{equation*}
R_{k}=B_{k}\setminus B_{k-1},\quad \chi _{k}=\chi _{R_{k}},\quad k\in 
\mathbb{Z}.
\end{equation*}

\begin{defn}
	\label{def:herz} Let $0<p,q\leq \infty $ and $\alpha \in \mathbb{R}$. The
	homogeneous Herz space $\dot{K}_{{p},q}^{{\alpha }}$ is defined as the set
	of all $f\in L_{\loc}^{{p}}\left( {\mathbb{R}^{n}}\setminus \{0\}\right) $
	such that 
	\begin{equation*}
	\big\|f\big\|_{\dot{K}_{{p},q}^{{\alpha }}}=\Big(\sum\limits_{k\in \mathbb{Z}%
	}2^{k{\alpha q}}\big\|f\,\chi _{k}\big\|_{{p}}^{q}\Big)^{1/q}<\infty
	\end{equation*}%
	(with the usual modifications when $q=\infty $).
\end{defn}

\begin{rem}
	Let $0<p,q\leq \infty $ and $\alpha \in \mathbb{R}$.$\newline
	\mathrm{(i)}$ The space $\dot{K}_{{p},p}^{{\alpha }}$\ coincides with the
	Lebesgue space $L^{p}(\mathbb{R}^{n},|\cdot |^{\alpha p })$. In addition%
	\begin{equation*}
	\dot{K}_{p,p}^{0}=L^{p}\text{.}
	\end{equation*}%
	$\mathrm{(ii)}$ Let $0<q_{1}\leq q_{2}\leq \infty $. Then%
	\begin{equation*}
	\dot{K}_{p,q_{1}}^{\alpha }\hookrightarrow \dot{K}_{p,q_{2}}^{\alpha }.
	\end{equation*}%
	$\mathrm{(iii)}$ The spaces $\dot{K}_{p,q}^{\alpha }$ are quasi-Banach
	spaces and if $\min (p,q)\geq 1$ then $\dot{K}_{p,q}^{\alpha }$ are Banach
	spaces.
\end{rem}

\begin{rem}
	A detailed discussion of the properties of Herz spaces my be found in \cite%
	{HerYang99} and \cite{LDH08}, and references therein.
\end{rem}

To present the definition of Herz-type Triebel-Lizorkin spaces,  we first need
the concept of a smooth dyadic resolution of unity. Let $\psi $\ be a
function\ in $\mathcal{S}(\mathbb{R}^{n})$\ satisfying%
\begin{equation*}
0\leq \psi \leq 1\text{\quad and}\quad \psi (x)=\left\{ 
\begin{array}{ccc}
1, & \text{if } & \left\vert x\right\vert \leq 1, \\ 
0, & \text{if} & \left\vert x\right\vert \geq \frac{3}{2}.%
\end{array}%
\right.
\end{equation*}%
We put $\mathcal{F}\varphi _{0}=\psi $, $\mathcal{F}\varphi _{1}=\psi (\frac{%
	\cdot }{2})-\psi $\ and $\mathcal{F}\varphi _{j}=\mathcal{F}\varphi
_{1}(2^{1-j}\cdot )\ $for$\ j=2,3$,.... Then $\{\mathcal{F}\varphi
_{j}\}_{j\in \mathbb{N}_{0}}$\ is a smooth dyadic resolution of unity, $%
\sum_{j=0}^{\infty }\mathcal{F}\varphi _{j}(x)=1$ for all $x\in \mathbb{R}%
^{n}$.\ Thus we obtain the Littlewood-Paley decomposition 
\begin{equation*}
f=\sum_{j=0}^{\infty }\varphi _{j}\ast f
\end{equation*}%
of all $f\in \mathcal{S}^{\prime }(\mathbb{R}^{n})$ $($convergence in $%
\mathcal{S}^{\prime }(\mathbb{R}^{n}))$.

We are now in a position to state the definition of Herz-type
Triebel-Lizorkin spaces.

\begin{defn}
	\label{Herz-Besov-Triebel}\textit{Let }$\alpha ,s\in \mathbb{R},0<p,q<\infty 
	$\textit{\ and }$0<\beta \leq \infty $\textit{. The Herz-type
		Triebel-Lizorkin space }$\dot{K}_{p,q}^{\alpha }F_{\beta }^{s}$ \textit{is
		the collection of all} $f\in \mathcal{S}^{^{\prime }}(\mathbb{R}^{n})$%
	\textit{\ such that}%
	\begin{equation*}
	\big\|f\big\|_{\dot{K}_{p,q}^{\alpha }F_{\beta }^{s}}=\Big\|\Big(%
	\sum\limits_{j=0}^{\infty }2^{js\beta }\left\vert \varphi _{j}\ast
	f\right\vert ^{\beta }\Big)^{1/\beta }\Big\|_{\dot{K}_{p,q}^{\alpha
	}}<\infty ,
	\end{equation*}%
	\textit{with the obvious modification if }$\beta =\infty .$
\end{defn}

\begin{rem}
	Let\ $s\in \mathbb{R},0<p,q<\infty ,0<\beta \leq \infty $ and $\alpha >-%
	\frac{n}{p}$. The spaces\ $\dot{K}_{p,q}^{\alpha }F_{\beta }^{s}$ are
	independent of the particular choice of the smooth dyadic resolution of
	unity\ $\{\mathcal{F}\varphi _{j}\}_{j\in \mathbb{N}_{0}}$(in the sense of\
	equivalent quasi-norms). In particular $\dot{K}_{p,q}^{\alpha }F_{\beta
	}^{s} $\ are quasi-Banach spaces and if $p,q,\beta \geq 1$, then they are
	Banach spaces. Further results, concerning, for instance , lifting
	properties, Fourier multiplier and local means characterizations can be
	found in \cite{Drihem1.13}-\cite{Drihem2.13}-\cite{Drmana}, \cite{XuYang05}
	and \cite{Xu05}.
\end{rem}

Now we give the definition of the spaces $F_{p,\beta }^{s}$.\vskip5pt

\begin{defn}
	\textit{Let }$s\in \mathbb{R},0<p<\infty $\textit{\ and }$0<\beta \leq
	\infty $\textit{. The Triebel-Lizorkin space }$F_{p,\beta }^{s}$ \textit{is
		the collection of all} $f\in \mathcal{S}^{\prime }(\mathbb{R}^{n})$\textit{\
		such that}%
	\begin{equation*}
	\big\|f\big\|_{F_{p,\beta }^{s}}=\Big\|\Big(\sum\limits_{j=0}^{\infty
	}2^{js\beta }\left\vert \varphi _{j}\ast f\right\vert ^{\beta }\Big)%
	^{1/\beta }\Big\|_{p}<\infty .
	\end{equation*}
\end{defn}

The theory of the spaces $F_{p,\beta }^{s}$ has been developed in detail in 
\cite{Sawano18}, \cite{Triebel83} and \cite{Triebel92} but has a longer
history already including many contributors; we do not want to discuss this
here. Clearly, for $s\in \mathbb{R},0<p<\infty $ and $0<\beta \leq \infty ,$ 
\begin{equation*}
\dot{K}_{p,p}^{0}F_{\beta }^{s}=F_{p,\beta }^{s}.
\end{equation*}

Let $w\in \mathcal{A}_{\infty }$, Muckenhoupt classes, $s\in \mathbb{R}$, $%
0<\beta \leq \infty $ and $0<p<\infty $. We define weighted Triebel-Lizorkin
space $F_{p,\beta }^{s}(\mathbb{R}^{n},w)$ to be the set of all
distributions $f\in \mathcal{S}^{\prime }(\mathbb{R}^{n})$ such that%
\begin{equation*}
\big\|f\big\|_{F_{p,\beta }^{s}(\mathbb{R}^{n},w)}=\Big\|\Big(%
\sum\limits_{j=0}^{\infty }2^{js\beta }\left\vert \varphi _{j}\ast
f\right\vert ^{\beta }\Big)^{1/\beta }\Big\|_{L^{p}(\mathbb{R}^{n},w)}
\end{equation*}%
is finite. In the limiting case $\beta =\infty $ the usual modification is
required.

The spaces $F_{p,\beta }^{s}(\mathbb{R}^{n},w)=F_{p,\beta }^{s}(w)$ are
independent of the particular choice of the smooth dyadic resolution of
unity $\{\mathcal{F}\varphi _{j}\}_{j\in \mathbb{N}_{0}}$ appearing in their
definitions. They are quasi-Banach spaces (Banach spaces for $p,q\geq 1$).
Moreover, for $w\equiv 1$ we obtain the usual (unweighted) Triebel-Lizorkin
spaces. We refer, in particular, to the papers \cite{Bui82} and \cite{IzSa12}
for a comprehensive treatment of  weighted  function spaces. Let $w_{\gamma }$ be a
power weight, i.e., $w_{\gamma }(x)=|x|^{\gamma }$ with $\gamma >-n$. Then
we have%
\begin{equation*}
F_{p,\beta }^{s}(w_{\gamma })=\dot{K}_{p,p}^{\frac{\gamma }{p}}F_{\beta
}^{s},
\end{equation*}%
in the sense of equivalent quasi-norms.

\begin{defn}
	$\mathrm{(i)}$ Let $1<p<\infty ,0<q<\infty ,-\frac{n}{p}<\alpha <n(1-\frac{1%
	}{p})$ and $s\in \mathbb{R}$. Then the Herz-type Bessel potential space $%
	\dot{k}_{p,s}^{\alpha ,q}$ is the\textit{\ }collection of all $f\in \mathcal{%
		S}^{\prime }(\mathbb{R}^{n})$\ such that%
	\begin{equation*}
	\big\|f\big\|_{\dot{k}_{p,s}^{\alpha ,q}}=\big\|(1+|\xi |^{2})^{\frac{s}{2}%
	}\ast f\big\|_{\dot{K}_{p,q}^{\alpha }}<\infty .
	\end{equation*}%
	$\mathrm{(ii)}$ Let $1<p<\infty ,0<q<\infty ,-\frac{n}{p}<\alpha <n(1-\frac{1%
	}{p})$ and $m\in \mathbb{N}$. The homogeneous Herz-type Sobolev space $\dot{W%
	}_{p,m}^{\alpha ,q}$ is the collection of all $f\in \mathcal{S}^{\prime }(%
	\mathbb{R}^{n})$\ such that%
	\begin{equation*}
	\big\|f\big\|_{\dot{W}_{p,m}^{\alpha ,q}}=\sum\limits_{|\beta |\leq m}\Big\|%
	\frac{\partial ^{\beta }f}{\partial x^{\beta }}\Big\|_{\dot{K}_{p,q}^{\alpha
	}}<\infty ,
	\end{equation*}%
	where the derivatives must be understood in the sense of distribution.
\end{defn}

In the following, we will present the connection between the Herz-type
Triebel-Lizorkin spaces and the Herz-type Bessel potential spaces; see \cite%
{LuYang97} and \cite{XuYang03}. Let $1<p,q<\infty $ and $-\frac{n}{p}<\alpha
<n(1-\frac{1}{p})$.$\ $If $s\in \mathbb{R}$, then%
\begin{equation*}
\dot{K}_{p,q}^{\alpha }F_{2}^{s}=\dot{k}_{p,s}^{\alpha ,q}
\end{equation*}%
with equivalent norms. If $s=m\in \mathbb{N}$, then

\begin{equation*}
\dot{K}_{p,q}^{\alpha }F_{2}^{m}=\dot{W}_{p,m}^{\alpha ,q}
\end{equation*}%
with equivalent norms. In particular%
\begin{equation*}
\dot{K}_{p,p}^{\alpha }F_{2}^{m}=W_{m}^{p}(\mathbb{R}^{n},|\cdot |^{\alpha
	p})\quad \text{(Sobolev spaces of power weights)}
\end{equation*}%
and%
\begin{equation}
\dot{K}_{p,p}^{0}F_{2}^{m}=W_{m}^{p}\quad \text{(Sobolev spaces),}\quad \dot{%
	K}_{p,q}^{\alpha }F_{2}^{0}=\dot{K}_{p,q}^{\alpha }.  \label{properties1}
\end{equation}%
Let $0<\theta <1,$ $0<p_{0}, p_{1},q_{0},q_{1}<\infty$, $0 <\beta_{0},\beta_{1} \leq\infty$ and $\alpha _{0}, \alpha _{1},s_{0}, s_{1}\in \mathbb{R}. $ We set%
\begin{equation*}
\frac{1}{p}=\frac{%
	1-\theta }{p_{0}}+\frac{\theta }{p_{1}},\quad \frac{1}{q}=\frac{1-\theta }{%
	q_{0}}+\frac{\theta }{q_{1}},\quad \frac{1}{\beta }=\frac{1-\theta }{\beta
	_{0}}+\frac{\theta }{\beta _{1}}
\end{equation*}%
and%
\begin{equation*}
\alpha =(1-\theta )\alpha _{0}+\theta \alpha _{1},\quad s=(1-\theta )s_{0}+\theta s_{1}.
\end{equation*}%
As an immediate consequence of H\"{o}lder's inequality we have the so-called
interpolation inequalities:%
\begin{equation}
\big\|f\big\|_{\dot{K}_{p,q}^{\alpha }F_{\beta }^{s}}\leq \big\|f\big\|_{%
	\dot{K}_{p_{0},q_{0},}^{\alpha _{0}}F_{\beta _{0}}^{s_{0}}}^{1-\theta }\big\|%
f\big\|_{\dot{K}_{p_{1},q_{1}}^{\alpha _{1}}F_{\beta _{1}}^{s_{1}}}^{\theta }
\label{Interpolation}
\end{equation}%
holds for all $f\in \dot{K}_{p_{0},q_{0}}^{\alpha _{0}}F_{\beta
	_{0}}^{s_{0}}\cap \dot{K}_{p_{1},q_{1}}^{\alpha _{1}}F_{\beta _{1}}^{s_{1}}$.

We collect some embeddings on these functions spaces as obtained in \cite%
{Drihem2.13}.

\begin{thm}
	\label{embeddings3}\textit{Let }$\alpha _{1},\alpha _{2},s_{1},s_{2}\in 
	\mathbb{R},0<s,p,q,r<\infty ,0<\beta \leq \infty ,\alpha _{1}>-\frac{n}{s}\ $%
	\textit{and }$\alpha _{2}>-\frac{n}{q}$. \textit{We suppose that }%
	\begin{equation*}
	s_{1}-\frac{n}{s}-\alpha _{1}=s_{2}-\frac{n}{q}-\alpha _{2}.
	\end{equation*}%
	\textit{Let }$0<q\leq s<\infty $ and $\alpha _{2}\geq \alpha _{1}$. The
	embedding%
	\begin{equation*}
	\dot{K}_{q,r}^{\alpha _{2}}F_{\infty }^{s_{2}}\hookrightarrow \dot{K}%
	_{s,p}^{\alpha _{1}}F_{\beta }^{s_{1}}
	\end{equation*}%
	holds if $0<r\leq p<\infty $.
\end{thm}

Let $0<p,q< \infty $. For later use, we introduce the following
abbreviations:%
\begin{equation*}
\sigma _{p}=n\max \Big(\frac{1}{p}-1,0\Big)\quad \text{and}\quad \sigma
_{p,q}=n\max \Big(\frac{1}{p}-1,\frac{1}{q}-1,0\Big).
\end{equation*}%
In the next we shall interpret $L_{\mathrm{loc}}^{1}$ as the set of regular
distributions, see \cite{Dr21}.

\begin{thm}
	\label{regular-distribution1 copy(1)}\textit{Let }$0<p,q<\infty ,0<\beta
	\leq \infty ,\alpha >-\frac{n}{p}$ and $s>\max (\sigma _{p},\frac{n}{p}%
	+\alpha -n)$. Then%
	\begin{equation*}
	\dot{K}_{p,q}^{\alpha }F_{\beta }^{s}\hookrightarrow L_{\mathrm{loc}}^{1}.
	\end{equation*}
\end{thm}

For any $a>0$, $f\in \mathcal{S}^{\prime }(\mathbb{R}^{n})$ and $x\in 
\mathbb{R}^{n}$, we denote, Peetre maximal function, 
\begin{equation*}
\varphi _{j}^{\ast ,a}f(x)=\sup_{y\in \mathbb{R}^{n}}\,\frac{\left\vert
	\varphi _{j}\ast f(y)\right\vert }{(1+2^{j}\left\vert x-y\right\vert )^{a}}%
,\quad j\in \mathbb{N}_{0}.
\end{equation*}%
We now present a fundamental characterization of the above spaces, which
plays an essential role in this paper, see \cite[Theorem 1]{Xu05}.

\begin{thm}
	\label{Peetre maximal function}Let $s\in \mathbb{R},0<p,q<\infty ,0<\beta
	\leq \infty $ and $\alpha >-\frac{n}{p}$. Let $a>\frac{n}{\min (p,\beta )}$.
	Then%
	\begin{equation*}
	\big\Vert f\big\Vert_{\dot{K}_{p,q}^{\alpha }F_{\beta }^{s}}^{\star }=%
	\Big\Vert\Big(\sum_{j=0}^{\infty }2^{js\beta }(\varphi _{j}^{\ast
		,a}f)^{\beta }\Big)^{1/\beta }\Big\Vert_{\dot{K}_{p,q}^{\alpha }},
	\end{equation*}%
	is an\ equivalent quasi-norm in $\dot{K}_{p,q}^{\alpha }F_{\beta }^{s}.$
\end{thm}

\section{Composition operators}

Let $G:\mathbb{R\rightarrow R}$ be a continuous function. To solve \eqref{int-equa1},
we study the action of the nonlinear function $G$ on Herz-type
Triebel-Lizorkin spaces. Let us recall some results obtained in \cite{Dr21.2}%
, where they proved for Triebel-Lizorkin spaces of power weights, but the
results can be easily expanded to Herz-type Triebel-Lizorkin spaces. Let $%
1<p,q<\infty ,0<\beta \leq \infty ,0\leq \alpha <n-\frac{n}{p}$. Let $T_{G}$
be a composition operator, or Nemytzkij operators, such that%
\begin{equation}
T_{G}(\mathbb{\dot{K}}_{p,q}^{\alpha }\mathbb{F}_{\beta }^{s})\subset 
\mathbb{\dot{K}}_{p,q}^{\alpha }\mathbb{F}_{\beta }^{s},  \label{Condition1}
\end{equation}%
where $\mathbb{\dot{K}}_{p,q}^{\alpha }\mathbb{F}_{\beta }^{s}$ is the
real-valued part of $\dot{K}_{p,q}^{\alpha }F_{\beta }^{s}$. If $s>\frac{n}{p%
}+\alpha $, then $G^{\prime }\in L_{\mathrm{loc}}^{\infty }(\mathbb{R})$ is
necessary. In the the case $0<s\leq \frac{n}{p}+\alpha $, we have\ $%
G^{\prime }\in L^{\infty }(\mathbb{R})$ is necessary.

Now, let $1<p,q<\infty ,0<\beta \leq \infty ,0\leq \alpha <n-\frac{n}{p}$
and $G\in C^{2}(\mathbb{R})$. Let $T_{G}$ be a composition operator with $%
\mathrm{\eqref{Condition1}}$ and%
\begin{equation*}
1+\frac{1}{p}<s<\frac{n}{p}+\alpha .
\end{equation*}%
Then $G(t)=ct$ for some constant $c$.

In this section we investigate sufficient conditions on $G$ such that %
\eqref{Condition1} holds. First we need the following lemma, which is
basically a consequence of Hardy's inequality in the sequence Lebesgue space 
$\ell _{q}$.

\begin{lem}
	\label{lq-inequality}\textit{Let }$0<a<1$\textit{\ and }$0<q\leq \infty $%
	\textit{. Let }$\left\{ \varepsilon _{k}\right\} _{k\in \mathbb{N}_{0}}$\textit{\ be a
		sequences of positive real numbers and denote }$\delta
	_{k}=\sum_{j=0}^{k}a^{k-j}\varepsilon _{j}$ and $\eta
	_{k}=\sum_{j=k}^{\infty }a^{j-k}\varepsilon _{j},k\in \mathbb{N}_{0}$. Then
	there exists a constant $c>0\ $\textit{depending only on }$a$\textit{\ and }$q$
	such that%
	\begin{equation*}
	\Big(\sum\limits_{k=0}^{\infty }\delta _{k}^{q}\Big)^{1/q}+\Big(%
	\sum\limits_{k=0}^{\infty }\eta _{k}^{q}\Big)^{1/q}\leq c\text{ }\Big(%
	\sum\limits_{k=0}^{\infty }\varepsilon _{k}^{q}\Big)^{1/q}.
	\end{equation*}
\end{lem}

As usual, we put%
\begin{equation*}
\mathcal{M(}f)(x)=\sup_{Q}\frac{1}{|Q|}\int_{Q}\left\vert f(y)\right\vert
dy,\quad f\in L_{\mathrm{loc}}^{1},
\end{equation*}%
where the supremum\ is taken over all cubes with sides parallel to the axis
and $x\in Q$. Also, we set $\mathcal{M}_{\sigma }(f)=\left( \mathcal{M(}%
\left\vert f\right\vert ^{\sigma })\right) ^{\frac{1}{\sigma }},0<\sigma
<\infty .$

Various important results have been proved in the space $\dot{K}%
_{p,q}^{\alpha }$ under some assumptions on $\alpha ,p$ and $q$. The
conditions $-\frac{n}{p}<\alpha <n(1-\frac{1}{p}),1<p<\infty $ and $0<q\leq
\infty $ is crucial in the study of the boundedness of classical operators
in $\dot{K}_{p,q}^{\alpha }$ spaces. This fact was first realized by Li and
Yang \cite{LiYang96} with the proof of the boundedness of the maximal
function. Some of our results of this paper are based on the following
result, see Tang and Yang \cite{TD00}.

\begin{lem}
	\label{Maximal-Inq}Let $1<\beta <\infty ,1<p<\infty $ and $0<q\leq \infty $.
	If  $\{f_{j}\}_{j\in \mathbb{N}_{0}}$ is a sequence of locally integrable functions
	on $\mathbb{R}^{n}$ and $-\frac{n}{p}<\alpha <n(1-\frac{1}{p})$, then%
	\begin{equation*}
	\Big\|\Big(\sum_{j=0}^{\infty }(\mathcal{M}(f_{j}))^{\beta }\Big)^{1/\beta }%
	\Big\|_{\dot{K}_{p,q}^{\alpha }}\leq c\Big\|\Big(\sum_{j=0}^{\infty
	}|f_{j}|^{\beta }\Big)^{1/\beta }\Big\|_{\dot{K}_{p,q}^{\alpha }}.
	\end{equation*}
\end{lem}

Let $\mu >0$ and $f\in L_{\mathrm{loc}}^{\max (1,\mu )}$. Define%
\begin{equation*}
I_{k}^{\mu }(f)(x)=\int_{\bar{B}_{-k}}|f(x+z)-f(x)|^{\mu }dz,\quad x\in 
\mathbb{R}^{n},k\in \mathbb{Z}.
\end{equation*}

\begin{lem}
	\label{Key-lemma1}Let $0<p,q<\infty ,0<\beta \leq \infty ,\alpha >-\frac{n}{p%
	}$ and 
	\begin{equation*}
	\max \Big(\sigma _{p,\beta },\frac{n}{p}+\alpha -n\Big)<s<\mu .
	\end{equation*}%
	Then there exists a constant $c>0 $ such that%
	\begin{equation}
	\Big\|\Big(\sum_{k=-\infty }^{\infty }2^{(n+s)k\beta }|I_{k}^{\mu
	}(f)|^{\beta }\Big)^{1/\beta }\Big\|_{\dot{K}_{p,q}^{\alpha }}\leq c\big\|f%
	\big\|_{\dot{K}_{p\mu ,q\mu }^{\frac{\alpha }{\mu }}F_{\beta \mu }^{\frac{s}{%
				\mu }}}^{\mu }  \label{key-estimate}
	\end{equation}%
	holds for all $f\in L_{\mathrm{loc}}^{\max (1,\mu )}$ with 
	\begin{equation*}
	f=\sum_{j=0}^{\infty }\varphi _{j}\ast f,
	\end{equation*}%
	in $L_{\mathrm{loc}}^{\mu }$, \textit{with the obvious modification if }$%
	\beta =\infty .$
\end{lem}

\begin{proof}
	We will do the proof in two steps.
	
	\textbf{Step 1. }We set $\Delta _{y}f(x)=f(x+y)-f(x),x,y\in \mathbb{R}^{n}$.
	A change of variable yields%
	\begin{equation*}
	2^{(n+s)k}I_{k}^{\mu }(f)(x)=2^{sk}\int_{\bar{B}_{0}}|\Delta
	_{z2^{-k}}f(x)|^{\mu }dz\lesssim J_{1,k}(f)(x)+J_{2,k}(f)(x)
	\end{equation*}%
	for all $x\in \mathbb{R}^{n}$, where the implicit constant is independent of 
	$x$ and $k$, 
	\begin{equation*}
	J_{1,k}(f)(x)=2^{sk}\int_{\bar{B}_{0}}\big|\sum_{j=0}^{k}\Delta
	_{z2^{-k}}(\varphi _{j}\ast f)(x)\big|^{\mu }dz
	\end{equation*}%
	and%
	\begin{equation*}
	J_{2,k}(f)(x)=2^{sk}\int_{\bar{B}_{0}}\big|\sum_{j=k+1}^{\infty }\Delta
	_{z2^{-k}}(\varphi _{j}\ast f)(x)\big|^{\mu }dz.
	\end{equation*}%
	\textit{Estimate of }$J_{1,k}$. Let $\Psi ,\Psi _{0}\in \mathcal{S}\left( 
	\mathbb{R}^{n}\right) $ be two functions such that $\mathcal{F}\Psi =1$ and $%
	\mathcal{F}\Psi _{0}=1$ on supp$\varphi _{1}$ and supp$\psi $, respectively.
	Using the mean value theorem we obtain for any $x\in \mathbb{R}^{n}$, $j\in 
	\mathbb{N}_{0}$ and $|z|\leq 1$%
	\begin{eqnarray*}
		\left\vert \Delta _{z2^{-k}}(\varphi _{j}\ast f)(x) \right\vert
		&=&\left\vert \Delta _{z2^{-k}}(\Psi _{j}\ast \varphi _{j}\ast
		f)(x)\right\vert \\
		&\leq &2^{-k}\sup_{\left\vert x-y\right\vert \leq c\text{ }%
			2^{-k}}\sum\limits_{|\beta |=1}\left\vert D^{\beta }(\Psi _{j}\ast \varphi
		_{j}\ast f)( y) \right\vert ,
	\end{eqnarray*}%
	with some positive constant $c$ independent of $x$, $j$ and $k$, and 
	\begin{equation*}
	\Psi _{j}(\cdot )=2^{(j-1)n}\Psi (2^{j-1}\cdot )\quad \text{for}\quad
	j=1,2,....
	\end{equation*}
	We see that if $|\beta |=1$ and $a>0$%
	\begin{eqnarray}
	&&\left\vert D^{\beta }(\Psi _{j}\ast \varphi _{j}\ast f)( y)
	\right\vert  \notag \\
	&=&2^{\left( j-1\right) n}\left\vert \int_{\mathbb{R}^{n}}D^{\beta }\left(
	\Psi \left( 2^{j-1}\left( y-z\right) \right) \right) \varphi _{j}\ast
	f( z) dz\right\vert  \notag \\
	&\leq &2^{\left( j-1\right) \left( n+1\right) }\int_{\mathbb{R}%
		^{n}}\left\vert \left( D^{\beta }\Psi \right) \left( 2^{j-1}\left(
	y-z\right) \right) \right\vert \left\vert \varphi _{j}\ast f( z)
	\right\vert dz.  \label{estimate-dir}
	\end{eqnarray}%
	The right-hand side in \eqref{estimate-dir} may be estimated as follows: 
	\begin{eqnarray*}
		&&c\text{ }2^{j\left( n+1\right) }\varphi _{j}^{\ast ,a}f( y)
		\int_{\mathbb{R}^{n}}\left\vert \left( D^{\beta }\Psi \right) \left(
		2^{j-1}\left( y-z\right) \right) \right\vert \left( 1+2^{j}\left\vert
		y-z\right\vert \right) ^{a}dz \\
		&\leq &c\text{ }2^{j}\varphi _{j}^{\ast ,a}f( y) .
	\end{eqnarray*}%
	Then we obtain for any $x\in \mathbb{R}^{n}$, $|z|\leq 1$ and any $j,k\in 
	\mathbb{N}_{0}$%
	\begin{eqnarray*}
		\left\vert \Delta _{z2^{-k}}(\varphi _{j}\ast f)( x) \right\vert
		&\leq &c\text{ }2^{j-k}\sup_{\left\vert x-y\right\vert \leq c\text{ }%
			2^{-k}}\varphi _{j}^{\ast ,a}f( y) \\
		&\leq &c\text{ }2^{j-k}\left( 1+2^{j-k}\right) ^{a}\sup_{\left\vert
			x-y\right\vert \leq c\text{ }2^{-k}}\frac{\varphi _{j}^{\ast ,a}f(
			y) }{\left( 1+2^{j}\left\vert x-y\right\vert \right) ^{a}} \\
		&\leq &c\text{ }2^{j-k}\varphi _{j}^{\ast ,a}f( x) ,
	\end{eqnarray*}%
	if $0\leq j\leq k,k\in \mathbb{N}_{0}$ and $x\in \mathbb{R}^{n}$. Therefore%
	\begin{equation*}
	J_{1,k}(f)(x)\lesssim 2^{sk}\big(\sum_{j=0}^{k}2^{j-k}\varphi _{j}^{\ast
		,a}f\left( x\right) \big)^{\mu },
	\end{equation*}%
	where the implicit constant is independent of $x$ and $k$, and this yields that%
	\begin{equation*}
	\Big\|\Big(\sum_{k=0}^{\infty }|J_{1,k}(f)|^{\beta }\Big)^{1/\beta }\Big\|_{%
		\dot{K}_{p,q}^{\alpha }}
	\end{equation*}%
	can be estimated by%
	\begin{equation*}
	c\Big\|\Big(\sum_{k=0}^{\infty }\big(\sum_{j=0}^{k}2^{(j-k)(1-\frac{s}{\mu }%
		)}2^{j\frac{s}{\mu }}\varphi _{j}^{\ast ,a}f\big)^{\mu \beta }\Big)^{1/\mu
		\beta }\Big\|_{\dot{K}_{p\mu ,q\mu }^{\frac{\alpha }{\mu }}}^{\mu }.
	\end{equation*}%
	Using Lemma \ref{lq-inequality} the last expression is bounded by%
	\begin{equation*}
	c\Big\|\Big(\sum_{k=0}^{\infty }\big(2^{k\frac{s}{\mu }}\varphi _{k}^{\ast
		,a}f\big)^{\mu \beta }\Big)^{1/\mu \beta }\Big\|_{\dot{K}_{p\mu ,q\mu }^{%
			\frac{\alpha }{\mu }}}^{\mu }\lesssim \big\|f\big\|_{\dot{K}_{p\mu ,q\mu }^{%
			\frac{\alpha }{\mu }}F_{\beta \mu }^\frac{s}{\mu } }^{\mu },
	\end{equation*}%
	where we have used Theorem \ref{Peetre maximal function}.
	
	\textit{Estimate of }$J_{2,k}$. We can distinguish two cases as follows:
	
	\noindent $\bullet $ \textit{Case 1. }$\min (p,\beta )>1$. Therefore $s>\max %
	\big(0,\frac{n}{p}+\alpha -n\big)$. Assume that $\alpha \geq n(1-\frac{1}{p}%
	) $. Let $1-\frac{s\min (p,\beta )}{n}<\lambda <\min(\frac{np}{n+\alpha p},\beta)$ be a
	strict positive real number,	which is possible because of 
	\begin{equation*}
	s>\frac{n}{p}+\alpha -n>\frac{np(\frac{n}{p}+\alpha -n)}{\min (p,\beta
		)(n+\alpha p)}=\frac{n}{\min (p,\beta )}\big(1-\frac{np}{n+\alpha p}\big).
	\end{equation*}%
	Let $\frac{n}{\mu \min (p,\beta )}<a<\frac{s}{\mu (1-\lambda )}$. Then 
	\begin{equation}
	\frac{s}{\mu }>a(1-\lambda ).  \label{choices-s}
	\end{equation}%
	If $-\frac{n}{p}<\alpha <n(1-\frac{1}{p})$, then we take $\lambda =1$. From
	this we deduce that for all $x\in \mathbb{R}^{n}$, $2^{-sk}J_{2,k}(f)(x)$ can be
	estimated by%
	\begin{eqnarray*}
		&&c\sum_{j=k+1}^{\infty }2^{(j-k)\varepsilon }\int_{\bar{B}_{0}}\big|%
		\Delta _{z2^{-k}}(\varphi _{j}\ast f)(x)\big|^{\mu }dz \\
		&\lesssim &\sum_{j=k+1}^{\infty }2^{(j-k)\varepsilon }\sup_{x\in \bar{B%
			}_{0}}\big|\Delta _{z2^{-k}}(\varphi _{j}\ast f)(x)\big|^{\mu (1-\lambda
			)}\int_{\bar{B}_{0}}\big|\Delta _{z2^{-k}}(\varphi _{j}\ast f)(x)\big|^{\mu
			\lambda }dz
	\end{eqnarray*}%
	where $0<\frac{2\varepsilon }{\mu }\leq \frac{s}{\mu }-a(1-\lambda )$ and the
	positive constant $c$ is independent of $k$ and $x$. Observe that%
	\begin{eqnarray*}
		&&\int_{\bar{B}_{0}}\big|\Delta _{z2^{-k}}(\varphi _{j}\ast f)(x)\big|^{\mu
			\lambda }dz \\
		&\lesssim &\big|\varphi _{j}\ast f(x)\big|^{\mu \lambda
		}+2^{kn}\int_{|y-x|\leq 2^{-k}}\big|\varphi _{j}\ast f(y)\big|^{\mu \lambda
		}dy \\
		&\lesssim &\big|\varphi _{j}\ast f(x)\big|^{\mu \lambda }+\mathcal{M}\big(%
		\big|\varphi _{j}\ast f\big|^{\mu \lambda }\big)(x).
	\end{eqnarray*}%
	This estimate combined with
	
	\begin{equation}
	\left\vert \Delta _{z2^{-k}}(\varphi _{j}\ast f)\left( x\right) \right\vert
	\leq c\text{ }2^{(j-k)a}\varphi _{j}^{\ast ,a}f\left( x\right)
	\label{Peetre}
	\end{equation}%
	for any $x\in \mathbb{R}^{n},|z|\leq 1$ and any $j\geq k+1$, yield%
	\begin{equation*}
	J_{2,k}(f)\lesssim J_{2,k,1}(f)+J_{2,k,2}(f),
	\end{equation*}%
	where 
	\begin{equation*}
	J_{2,k,1}(f)=\sum_{j=k+1}^{\infty }2^{(j-k)(\varepsilon +a\mu (1-\lambda
		)-s)}\big(2^{j\frac{s}{\mu }}\varphi _{j}^{\ast ,a}f\big)^{\mu (1-\lambda )}%
	\big|2^{j\frac{s}{\mu }}\varphi _{j}\ast f\big|^{\mu \lambda }
	\end{equation*}%
	and%
	\begin{equation*}
	J_{2,k,2}(f)=\sum_{j=k+1}^{\infty }2^{(j-k)(\varepsilon +a\mu (1-\lambda
		)-s)}\big(2^{j\frac{s}{\mu }}\varphi _{j}^{\ast ,a}f\big)^{\mu (1-\lambda )}%
	\mathcal{M}\big(2^{j\frac{s}{\mu }}|\varphi _{j}\ast f|\big)^{\mu \lambda }.
	\end{equation*}%
	By similarity we estimate only $J_{2,k,2}(f)$. Using Lemma \ref%
	{lq-inequality} and H\"{o}lder's inequality we get%
	\begin{eqnarray*}
		&&\Big(\sum_{k=0}^{\infty }(J_{2,k,2}(f))^{\beta }\Big)^{1/\beta } \\
		&\lesssim &\Big(\sum_{k=0}^{\infty }\big(2^{k\frac{s}{\mu }}\varphi
		_{k}^{\ast ,a}f\big)^{\mu (1-\lambda )\beta }\big(\mathcal{M}\big(2^{k\frac{s%
			}{\mu }}|\varphi _{k}\ast f|\big)^{\mu \lambda }\big)^{\beta }\Big)^{1/\beta
		} \\
		&\lesssim &\Big(\sum_{k=0}^{\infty }\big(2^{k\frac{s}{\mu }}\varphi
		_{k}^{\ast ,a}f\big)^{\mu \beta }\Big)^{(1-\lambda )/\beta }\Big(%
		\sum_{k=0}^{\infty }\big(\mathcal{M}\big(2^{k\frac{s}{\mu }}|\varphi
		_{k}\ast f|\big)^{\mu \lambda }\big)^{\beta /\lambda }\Big)^{\lambda /\beta
		}.
	\end{eqnarray*}%
	Again by H\"{o}lder's inequality%
	\begin{equation*}
	\Big\|\Big(\sum_{k=0}^{\infty }(J_{2,k,2}(f))^{\beta }\Big)^{1/\beta }\Big\|%
	_{\dot{K}_{p,q}^{\alpha }},
	\end{equation*}%
	can be estimated by%
	\begin{eqnarray*}
		&&c\Big\|\Big(\sum_{k=0}^{\infty }\big(2^{k\frac{s}{\mu }}\varphi _{k}^{\ast
			,a}f\big)^{\mu \beta }\Big)^{1/\beta }\Big\|_{\dot{K}_{p,q}^{\alpha
		}}^{1-\lambda } \\
		&&\times \Big\|\Big(\sum_{k=0}^{\infty }\big(\mathcal{M}\big(2^{k\frac{s}{%
				\mu }}|\varphi _{k}\ast f|\big)^{\mu \lambda }\big)^{\beta /\lambda }\Big)%
		^{\lambda /\beta }\Big\|_{\dot{K}_{\frac{p}{\lambda },\frac{q}{\lambda }%
			}^{\alpha \lambda }} \\
		&\lesssim &\big\|f\big\|_{\dot{K}_{p\mu ,q\mu }^{\frac{\alpha }{\mu }%
			}F_{\beta \mu }^{\frac{s}{\mu }}}^{(1-\lambda )\mu }\Big\|\Big(%
		\sum_{k=0}^{\infty }\big(2^{k\frac{s}{\mu }}|\varphi _{k}\ast f|\big)^{\mu
			\beta }\Big)^{1/\mu \beta }\Big\|_{\dot{K}_{p\mu ,q\mu }^{\frac{\alpha }{\mu 
		}}}^{\lambda \mu }
	\end{eqnarray*}%
	where we have used Theorem \ref{Peetre maximal function} and Lemma \ref%
	{Maximal-Inq}. Obviously we can estimate the last term by%
	\begin{equation*}
	c\big\|f\big\|_{\dot{K}_{p\mu ,q\mu }^{\frac{\alpha }{\mu }}F_{\beta \mu }^{%
			\frac{s}{\mu }}}^{\mu }.
	\end{equation*}%
	$\bullet $\textit{\ Case 2. }$\min (p,\beta )\leq 1$. If $-\frac{n}{p}%
	<\alpha <n(1-\frac{1}{p})$, then $s>\frac{n}{\min (p,\beta )}-n$. Taking $\max(0,1-%
	\frac{s\min (p,\beta )}{n})<\lambda <\min (1,p,\beta )$. The same arguments as
	in Case 1 yield the desired estimate. Now assume that $\alpha \geq n(1-\frac{%
		1}{p})$. Therefore 
	\begin{equation*}
	s>\max \Big(\frac{n}{\min (p,\beta )}-n,\frac{n}{p}+\alpha -n\Big).
	\end{equation*}%
	Taking $\max(0,1-\frac{s\min (p,\beta )}{n})<\lambda <\min (p,\frac{np}{n+\alpha p}%
	,\beta )$. The desired estimate can be done in the same manner as in Case 1.
	
	\textbf{Step 2.} We will estimate%
	\begin{equation*}
	\Big\|\Big(\sum_{k=-\infty }^{-1}2^{(n+s)k\beta }|I_{k}^{\mu }(f)|^{\beta }%
	\Big)^{1/\beta }\Big\|_{\dot{K}_{p,q}^{\alpha }}.
	\end{equation*}%
	We employ the same notations as in Step 1. Recall that%
	\begin{equation*}
	f=\sum_{j=0}^{\infty }\varphi _{j}\ast f.
	\end{equation*}%
	Define%
	\begin{equation*}
	M_{k,2}(f)(x)=\int_{\bar{B}_{0}}\big|\sum_{j=0}^{\infty }\Delta
	_{z2^{-k}}(\varphi _{j}\ast f)(x)\big|^{\mu }dz.
	\end{equation*}%
	As in the estimation of $J_{2,k}$, we obtain 
	\begin{equation*}
	M_{2,k}(f)\lesssim M_{2,k,1}(f)+M_{2,k,2}(f),
	\end{equation*}%
	where 
	\begin{equation*}
	M_{2,k,1}(f)=2^{-ka\mu (1-\lambda )}\sum_{j=0}^{\infty }2^{j(\varepsilon
		+a\mu (1-\lambda )-s)}\big(2^{j\frac{s}{\mu }}\varphi _{j}^{\ast ,a}f\big)%
	^{\mu (1-\lambda )}\big|2^{j\frac{s}{\mu }}\varphi _{j}\ast f\big|^{\mu
		\lambda }
	\end{equation*}%
	and%
	\begin{equation*}
	M_{2,k,2}(f)=2^{-ka\mu (1-\lambda )}\sum_{j=0}^{\infty }2^{j(\varepsilon
		+a\mu (1-\lambda )-s)}\big(2^{j\frac{s}{\mu }}\varphi _{j}^{\ast ,a}f\big)%
	^{\mu (1-\lambda )}\mathcal{M}\big(2^{j\frac{s}{\mu }}|\varphi _{j}\ast f|%
	\big)^{\mu \lambda },
	\end{equation*}%
	with the help of $\mathrm{\eqref{Peetre}}$\textrm{.} By similarity we
	estimate only $M_{2,k,2}$. Obviously%
	\begin{equation*}
	M_{2,k,2}(f)\lesssim 2^{-ka\mu (1-\lambda )}\sup_{j\in\mathbb{N}_{0}}\Big(\big(2^{j%
		\frac{s}{\mu }}\varphi _{j}^{\ast ,a}f\big)^{\mu (1-\lambda )}\mathcal{M}%
	\big(2^{j\frac{s}{\mu }}|\varphi _{j}\ast f|\big)^{\mu \lambda }\Big)
	\end{equation*}%
	and this yields that%
	\begin{equation*}
	\Big(\sum_{k=-\infty }^{-1}2^{sk\beta }|M_{2,k,2}|^{\beta }\Big)^{1/\beta
	}\lesssim \sup_{j\in\mathbb{N}_{0}}\Big(\big(2^{j\frac{s}{\mu }}\varphi _{j}^{\ast ,a}f%
	\big)^{\mu (1-\lambda )}\mathcal{M}\big(2^{j\frac{s}{\mu }}|\varphi _{j}\ast
	f|\big)^{\mu \lambda }\Big).
	\end{equation*}%
	By the same arguments as used in Step 1 we obtain the desired estimate. The
	proof is complete.
\end{proof}

Now we present the case of $s=\mu $, where the proof is very similar to
Lemma \ref{Key-lemma1}.

\begin{lem}
	\label{Key-lemma2}Let $0<p,q<\infty ,\alpha >-\frac{n}{p}$ and 
	\begin{equation*}
	\max \Big(\sigma _{p},\frac{n}{p}+\alpha -n\Big)<\mu .
	\end{equation*}%
	Then there exists a positive  constant $c$ such that%
	\begin{equation*}
	\Big\|\sup_{k\in \mathbb{Z}}2^{(n+\mu )k}|I_{k}^{\mu }(f)|\Big\|_{\dot{K}%
		_{p,q}^{\alpha }}\leq c\big\|f\big\|_{\dot{K}_{p\mu ,q\mu }^{\frac{\alpha }{%
				\mu }}F_{1}^{1}}^{\mu }
	\end{equation*}%
	holds for all $f\in L_{\mathrm{loc}}^{\max (1,\mu )}$ with 
	\begin{equation*}
	f=\sum_{j=0}^{\infty }\varphi _{j}\ast f,
	\end{equation*}%
	in $L_{\mathrm{loc}}^{\mu }.$
\end{lem}

Using the fact that $\big\|f\big\|_{\dot{K}_{p\mu ,q\mu }^{\frac{\alpha }{%
			\mu }}F_{\beta \mu }^{\frac{s}{\mu }}}^{\mu }\leq \big\|f\big\|_{\dot{K}%
	_{p,q}^{\alpha }F_{\beta }^{s}}\big\|f\big\|_{\infty }^{\mu -1}$, we we
immediately arrive at the following results.

\begin{lem}
	\label{Key-lemma1 copy(1)}Let $0<p,q<\infty ,0<\beta \leq \infty ,\alpha >-%
	\frac{n}{p}$ and 
	\begin{equation*}
	\max \Big(1,\sigma _{p,\beta },\frac{n}{p}+\alpha -n\Big)<s<\mu .
	\end{equation*}%
	Then there exists a positive  constant $c$ such that%
	\begin{equation*}
	\Big\|\Big(\sum_{k=-\infty }^{\infty }2^{(n+s)k\beta }|I_{k}^{\mu
	}(f)|^{\beta }\Big)^{1/\beta }\Big\|_{\dot{K}_{p,q}^{\alpha }}\leq c\big\|f%
	\big\|_{\dot{K}_{p,q}^{\alpha }F_{\beta }^{s}}\big\|f\big\|_{\infty }^{\mu
		-1}
	\end{equation*}%
	holds for all $f\in \dot{K}_{p,q}^{\alpha }F_{\beta }^{s}\cap L^{\infty }$.
\end{lem}

\begin{rem}
	Corresponding statements to Lemmas \ref{Key-lemma1}, \ref{Key-lemma2}\ and %
	\ref{Key-lemma1 copy(1)} were proved by Runst \cite[Lemma 1]{Ru86}, with $%
	\alpha =0,p=q$ and the case of bounded functions, while with $\alpha =0,p=q$
	has been given by Sickel in \cite[Lemmas 1,2]{Si89} . In our proof we have
	used the ideas of \cite[Lemmas 1, 2]{Si89}.
\end{rem}

The next two lemmas are used in the proof of our results, see e.g. \cite%
{BouDr21}.

\begin{lem}
	\label{Key-lemma4}Let\textit{\ }$s\in \mathbb{R},A,B>0,0<p,q<\infty ,0<\beta
	\leq \infty $\textit{\ }and\ $\alpha >-\frac{n}{p}$. Let $\left\{
	f_{l}\right\} _{l\in \mathbb{N}_{0}}$ be a sequence of functions\ such that 
	\begin{equation*}
	\mathrm{supp}\mathcal{F}f_{0}\subseteq \left\{ \xi \in \mathbb{R}%
	^{n}:\left\vert \xi \right\vert \leq A\right\}
	\end{equation*}%
	and%
	\begin{equation*}
	\mathrm{supp}\mathcal{F}f_{l}\subseteq \left\{ \xi \in \mathbb{R}%
	^{n}:B2^{l+1}\leq \left\vert \xi \right\vert \leq A2^{l+1}\right\} .
	\end{equation*}%
	There exists a constant $c>0$ such that the following inequality%
	\begin{equation*}
	\Big\Vert\sum_{l=0}^{\infty }\,f_{l}\Big\Vert_{\dot{K}_{p,q}^{\alpha
		}F_{\beta }^{s}}\leq c\Big\Vert\Big(\sum_{l=0}^{\infty }2^{ls\beta
	}\left\vert f_{l}\right\vert ^{\beta }\Big)^{1/\beta }\Big\Vert_{\dot{K}%
		_{p,q}^{\alpha }}
	\end{equation*}%
	holds.
\end{lem}

\begin{lem}
	\label{Key-lemma5}Let\textit{\ }$A,B>0,0<p,q<\infty ,0<\beta \leq \infty $%
	\textit{\ }and\ $\alpha >-\frac{n}{p}$. Let $s>\max (\sigma _{p},\frac{n}{p}%
	+\alpha -n)$. Let $\left\{ f_{l}\right\} _{l\in \mathbb{N}_{0}}$ be a
	sequence of functions\ such that 
	\begin{equation*}
	\mathrm{supp}\mathcal{F}f_{l}\subseteq \left\{ \xi \in \mathbb{R}%
	^{n}:\left\vert \xi \right\vert \leq A2^{l+1}\right\} .
	\end{equation*}%
	\textit{T}hen it holds that 
	\begin{equation*}
	\Big\Vert\sum_{l=0}^{\infty }\,f_{l}\Big\Vert_{\dot{K}_{p,q}^{\alpha
		}F_{\beta }^{s}}\leq c\Big\Vert\Big(\sum_{l=0}^{\infty }2^{ls\beta
	}\left\vert f_{l}\right\vert ^{\beta }\Big)^{1/\beta }\Big\Vert_{\dot{K}%
		_{p,q}^{\alpha }}.
	\end{equation*}
\end{lem}

Let $G:\mathbb{R}\rightarrow $ $\mathbb{R}$ be a continuous function. We
shall deal with sufficient conditions on $G$ to guarantee an embedding 
\begin{equation*}
T_{G}(\dot{K}_{p,q}^{\alpha }F_{\beta }^{s})=G(\dot{K}_{p,q}^{\alpha
}F_{\beta }^{s})\subset \dot{K}_{p,q}^{\alpha }F_{\beta }^{s}.
\end{equation*}%
First we begin with the case where $G$ is polynomial.

\begin{thm}
	\label{Multiplication1}Let $0<p,q<\infty ,0<\beta \leq \infty $, $s\geq 
	\frac{n}{p}-\frac{n}{q},\alpha \geq 0$ and 
	\begin{equation}
	\max \Big(0,\frac{n}{p}+\alpha -\frac{n}{m}\Big)<s<\frac{n}{p}+\alpha ,\quad
	m=2,3,....  \label{hypothesis}
	\end{equation}%
	We put%
	\begin{equation*}
	s_{m}=s-(m-1)\big(\frac{n}{p}+\alpha -s\big).
	\end{equation*}%
	Then%
	\begin{equation}
	\big\|f^{m}\big\|_{\dot{K}_{p,q}^{\alpha }F_{\beta }^{s_{m}}}\lesssim \big\|f%
	\big\|_{\dot{K}_{p,q}^{\alpha }F_{\beta }^{s}}^{m}  \label{Key-product}
	\end{equation}%
	holds for all $f\in \dot{K}_{p,q}^{\alpha }F_{\beta }^{s}.$
\end{thm}

\begin{proof}
	We will do the proof into three steps.
	
	\textbf{Step 1.}\textit{\ Preparation.} Let $\{\mathcal{F}\varphi
	_{j}\}_{j\in \mathbb{N}_{0}}$ be a partition of unity and $f\in \mathcal{S}%
	^{\prime }(\mathbb{R}^{n})$. We define the convolution operators $\Delta
	_{j} $\ by the following: 
	\begin{equation*}
	\Delta _{j}f=\varphi _{j}\ast f,\quad j\in \mathbb{N}\quad \text{and}\quad
	\Delta _{0}f=\varphi _{0}\ast f=\mathcal{F}^{-1}\psi \ast f.
	\end{equation*}%
	We define the convolution operators $Q_{j}$, $j\in \mathbb{N}_{0}$ by the
	following: 
	\begin{equation*}
	Q_{j}f=\mathcal{F}^{-1}\psi _{j}\ast f,\quad j\in \mathbb{N}_{0},
	\end{equation*}%
	where $\mathcal{F}^{-1}\psi _{j}=2^{jn}\mathcal{F}^{-1}\psi (2^{j}\cdot )$
	and we see that 
	\begin{equation*}
	Q_{j}f=\sum_{k=0}^{j}\Delta _{k}f,\quad j\in \mathbb{N}_{0}.
	\end{equation*}%
	For all $f_{i}\in \mathcal{S}^{\prime }(\mathbb{R}^{n})$, $i=1,2,...,m$ the
	product $\prod_{i=1}^{m}f_{i}$ is defined by 
	\begin{equation*}
	\prod_{i=1}^{m}f_{i}=\lim_{j\rightarrow \infty }\prod_{i=1}^{m}Q_{j}f_{i},
	\end{equation*}%
	if the limit on the right-hand side exists in $\mathcal{S}^{\prime }(\mathbb{%
		R}^{n})$. The following decomposition of this product is given in \cite[%
	Chapter 4]{RS96}. We have the following formal decomposition:%
	\begin{equation*}
	\prod_{i=1}^{m}f_{i}=\sum_{k_{1},...,k_{m}=0}^{\infty }\prod_{i=1}^{m}\left(
	\Delta _{k_{i}}f_{i}\right) .
	\end{equation*}%
	The fundamental idea is to split $\prod_{i=1}^{m}f_{i}$ into two parts, both
	of them being always defined. Let $N$ be a natural number greater than $%
	1+\log _{2}3\left( m-1\right) $. Then we have the following decomposition:%
	\begin{eqnarray*}
		\prod_{i=1}^{m}f_{i} &=&\sum_{j=0}^{\infty }\left[ Q_{j-N}f_{1}\cdot
		...\cdot Q_{j-N}f_{m-1}\cdot \Delta _{j}f_{m}+...\right. \\
		&&\left. +\left( \Pi _{l\neq k}Q_{j-N}f_{l}\right) \Delta
		_{k}f_{j}+...+\Delta _{j}f_{1}\cdot Q_{j-N}f_{2}\cdot ....\cdot Q_{j-N}f_{m} 
		\right] \\
		&&+\sum_{j=0}^{\infty }\sum^{j}\left( \Delta _{k_{1}}f_{1}\right) \cdot
		....\cdot \left( \Delta _{k_{m}}f_{m}\right) ,
	\end{eqnarray*}%
	where the $\sum^{j}$ is taken over all $k\in \mathbb{Z}_{+}^{n}$ such that 
	\begin{equation*}
	\max_{\ell =1,...,m}k_{1}=k_{k_{m_{0}}}=j\quad \text{and}\quad \max_{\ell
		\neq m_{0}}\left\vert \ell -k_{\ell }\right\vert <N.
	\end{equation*}%
	Of course, if $k<0$ we put $\Delta _{k}f=0$. Probably $\sum^{j}$ becomes
	more transparent by restricting to a typical part, which can be taken to be 
	\begin{equation*}
	\Big(\prod_{i\in I_{1}}\Delta _{j}f_{i}\Big)\prod_{i\in I_{2}}Q_{_{j}}f_{i},
	\end{equation*}%
	where 
	\begin{equation*}
	I_{1},I_{2}\subset \left\{ 1,...,m\right\} ,\quad I_{1}\cap I_{2}=\emptyset
	,\quad I_{1}\cup I_{2}=\left\{ 1,...,m\right\} =I,\quad \left\vert
	I_{1}\right\vert \geq 2.
	\end{equation*}%
	We introduce the following notations%
	\begin{equation*}
	\Pi _{1,k}(f_{1},f_{2},...,f_{m})=\sum_{j=N}^{\infty }\Big(\prod_{i\neq
		k}Q_{_{j-N}}f_{i}\Big)\Delta _{j}f_{k}
	\end{equation*}%
	and%
	\begin{equation*}
	\Pi _{2}(f_{1},f_{2},...,f_{m})=\sum_{j=0}^{\infty }\sum^{j}\Big(%
	\prod_{i=1}^{m}\Delta _{k_{i}}f_{i}\Big).
	\end{equation*}%
	The advantage of the above decomposition is based on%
	\begin{equation*}
	\text{supp }\mathcal{F}\Big(\Big(\prod_{i\neq k}Q_{_{j-N}}f_{i}\Big)\Delta
	_{j}f_{k}\Big)\subset \big\{\xi \in \mathbb{R}^{n}:2^{j-1}\leq \left\vert
	\xi \right\vert \leq 2^{j+1}\big\},\quad j\geq N
	\end{equation*}%
	and%
	\begin{equation*}
	\text{supp }\mathcal{F}\Big(\sum^{j}\Big(\prod_{i=1}^{m}\Delta _{k_{i}}f_{i}%
	\Big)\Big)\subset \big\{\xi \in \mathbb{R}^{n}:\left\vert \xi \right\vert
	\leq 2^{j+N-2}\big\},\quad {j\in\mathbb{N}_{0}}.
	\end{equation*}
	
	\textbf{Step 2.} We will prove \eqref{Key-product}. Observe that we need
	only to estimate%
	\begin{equation*}
	\Pi _{1}(f,f,...,f)=\sum_{j=N}^{\infty }(Q_{_{j-N}}f)^{m-1}\Delta _{j}f
	\end{equation*}%
	and%
	\begin{equation*}
	\Pi _{2}(f,f,...,f)=\sum_{j=0}^{\infty }(\Delta
	_{j}f)^{|I_{1}|}(Q_{_{j}}f)^{|I_{2}|}.
	\end{equation*}%
	Define%
	\begin{equation*}
	\frac{1}{v}=\frac{1}{p}+(m-1)\Big(\frac{1}{p}-\frac{s}{n}\Big).
	\end{equation*}%
	Therefore we have the following Sobolev embeddings%
	\begin{equation*}
	\dot{K}_{v,q}^{\alpha m}F_{\beta }^{s}\hookrightarrow \dot{K}_{p,q}^{\alpha
	}F_{\beta }^{s_{m}}.
	\end{equation*}%
	Lemma \ref{Key-lemma4} gives%
	\begin{equation*}
	\big\|\Pi _{1}(f,f,...,f)\big\|_{\dot{K}_{v,q}^{\alpha m}F_{\beta }^{s}}
	\end{equation*}%
	can be estimated by%
	\begin{eqnarray*}
		&&c\Big\|\Big(\sum_{j=N}^{\infty }|2^{js}(Q_{_{j-N}}f)^{m-1}\Delta
		_{j}f|^{\beta }\Big)^{\frac{1}{\beta }}\Big\|_{\dot{K}_{v,q}^{\alpha m}} \\
		&\lesssim &\Big\|(\sup_{j\geq N}|Q_{_{j-N}}f|)^{m-1}\Big(\sum_{j=N}^{\infty
		}|2^{js}\Delta _{j}f|^{\beta }\Big)^{\frac{1}{\beta }}\Big\|_{\dot{K}%
			_{v,q}^{\alpha m}}.
	\end{eqnarray*}%
	By H\"{o}lder's inequality we estimate the last term by%
	\begin{equation*}
	c\big\|\sup_{j\geq N}|Q_{_{j-N}}f|\big\|_{\dot{K}_{b,\infty }^{\alpha
	}}^{m-1}\big\|f\big\|_{\dot{K}_{p,q}^{\alpha }F_{\beta }^{s}},
	\end{equation*}%
	with $\frac{1}{b}=\frac{1}{p}-\frac{s}{n}$. Recall that 
	\begin{eqnarray}
	\big\Vert\sup_{j\geq N}|Q_{_{j-N}}f|\big\Vert_{\dot{K}_{b,\infty }^{\alpha
	}} &\lesssim &\big\Vert\sup_{j\geq N}|Q_{_{j-N}}f|\big\Vert_{\dot{K}%
		_{b,b}^{\alpha }}  \notag \\
	&\lesssim &\big\Vert f\big\Vert_{F_{b,2}^{0}(\mathbb{R}^{n},|\cdot |^{\alpha
			b})}  \notag \\
	&\lesssim &\big\Vert f\big\Vert_{\dot{K}_{b,b}^{\alpha }F_{2}^{0}},
	\label{Hardy1}
	\end{eqnarray}%
	see \cite[Theorem 1.4]{Bui82}, because of $-\frac{n}{b}<\alpha <n(1-\frac{1}{%
		b})$. Since, $s\geq \frac{n}{p}-\frac{n}{q}$, thanks to the embedding%
	\begin{equation}
	\dot{K}_{p,q}^{\alpha }F_{\beta }^{s}\hookrightarrow \dot{K}_{b,b}^{\alpha
	}F_{2}^{0},  \label{embedding1}
	\end{equation}%
	see Theorem \ref{embeddings3}, we obtain 
	\begin{equation*}
	\big\|\Pi _{1}(f,f,...,f)\big\|_{\dot{K}_{v,q}^{\alpha m}F_{\beta
		}^{s}}\lesssim \big\|f\big\|_{\dot{K}_{p,q}^{\alpha }F_{\beta }^{s}}^{m}.
	\end{equation*}%
	Now we estimate $\Pi _{2}(f,f,...,f)$. Define%
	\begin{equation*}
	\frac{1}{u}=\frac{|I_{1}|}{p}+\frac{|I_{2}|}{b},\quad \sigma -\frac{n}{u}%
	-\alpha m=s_{m}-\frac{n}{p}-\alpha .
	\end{equation*}%
	Observe that $\sigma =|I_{1}|s$. Hence%
	\begin{equation*}
	\dot{K}_{u,\frac{q}{|I_{1}|}}^{\alpha m}F_{\frac{\beta }{|I_{1}|}%
	}^{|I_{1}|s}\hookrightarrow \dot{K}_{p,q}^{\alpha }F_{\beta }^{s_{m}}.
	\end{equation*}%
	From \eqref{hypothesis} it follows that $\sigma >\max \Big(0,\frac{n}{u}%
	+\alpha m-n\Big)$. Lemma \ref{Key-lemma5} gives%
	\begin{eqnarray*}
		&&\big\|\Pi _{2}(f,f,...,f)\big\|_{\dot{K}_{u,\frac{q}{|I_{1}|}}^{\alpha
				m}F_{\frac{\beta }{|I_{1}|}}^{|I_{1}|s}} \\
		&\lesssim &\Big\|\Big(\sum_{j=0}^{\infty
		}|2^{j|I_{1}|s}(Q_{j}f)^{|I_{2}|}(\Delta _{j}f)^{|I_{1}|}|^{\frac{\beta }{%
				|I_{1}|}}\Big)^{\frac{|I_{1}|}{\beta }}\Big\|_{\dot{K}_{u,\frac{q}{|I_{1}|}%
			}^{\alpha m}} \\
		&\lesssim &\Big\|(\sup_{j\geq 0}|Q_{j}f|)^{|I_{2}|}\Big(\sum_{j=0}^{\infty
		}|2^{js}\Delta _{j}f|^{\beta }\Big)^{\frac{|I_{1}|}{\beta }}\Big\|_{\dot{K}%
			_{u,\frac{q}{|I_{1}|}}^{\alpha m}}.
	\end{eqnarray*}%
	Again, by H\"{o}lder's inequality we estimate the last term by%
	\begin{equation*}
	c\big\|\sup_{j\geq N}|Q_{_{j}}f|\big\|_{\dot{K}_{b,\infty }^{\alpha
	}}^{|I_{2}|}\Big\|\Big(\sum_{j=0}^{\infty }|2^{js}\Delta _{j}f|^{\beta }\Big)%
	^{\frac{1}{\beta }}\Big\|_{\dot{K}_{p,q}^{\alpha }}^{|I_{1}|}\lesssim \big\|f%
	\big\|_{\dot{K}_{p,q}^{\alpha }F_{\beta }^{s}}^{m},
	\end{equation*}%
	where we have used \eqref{Hardy1} and \eqref{embedding1}.
\end{proof}

\begin{thm}
	\label{Multiplication1 copy(1)}Let $0<p,q<\infty ,0<\beta \leq \infty
	,\alpha \geq 0$ and 
	\begin{equation*}
	s>\max \Big(0,\frac{n}{p}+\alpha -n\Big),\quad m=2,3,....
	\end{equation*}%
	Then%
	\begin{equation*}
	\big\|f^{m}\big\|_{\dot{K}_{p,q}^{\alpha }F_{\beta }^{s}}\lesssim \big\|f%
	\big\|_{\dot{K}_{p,q}^{\alpha }F_{\beta }^{s}}\big\|f\big\|_{\infty }^{m-1}
	\end{equation*}%
	holds for all $f\in \dot{K}_{p,q}^{\alpha }F_{\beta }^{s}\cap L^{\infty }.$
\end{thm}

\begin{proof}
	First, we estimate $\Pi _{1}(f,f,...,f)$. Recall that%
	\begin{equation}
	\sup_{j\in \mathbb{N}_{0}}|Q_{_{j}}f|\lesssim \big\|f\big\|_{\infty }\quad 
	\text{and}\quad \sup_{j\in \mathbb{N}_{0}}|\Delta _{j}f|\lesssim \big\|f%
	\big\|_{\infty }.  \label{help1}
	\end{equation}%
	Lemma \ref{Key-lemma4} gives%
	\begin{equation*}
	\big\|\Pi _{1}(f,f,...,f)\big\|_{\dot{K}_{p,q}^{\alpha }F_{\beta }^{s}}
	\end{equation*}%
	can be estimated by%
	\begin{eqnarray*}
		&&c\Big\|\Big(\sum_{j=N}^{\infty }|2^{js}(Q_{_{j-N}}f)^{m-1}\Delta
		_{j}f|^{\beta }\Big)^{\frac{1}{\beta }}\Big\|_{\dot{K}_{p,q}^{\alpha }} \\
		&\lesssim &\Big\|(\sup_{j\geq N}|Q_{_{j-N}}f|)^{m-1}\Big(\sum_{j=N}^{\infty
		}|2^{js}\Delta _{j}f|^{\beta }\Big)^{\frac{1}{\beta }}\Big\|_{\dot{K}%
			_{p,q}^{\alpha }} \\
		&\lesssim &\big\|f\big\|_{\infty }^{m-1}\big\|f\big\|_{\dot{K}_{p,q}^{\alpha
			}F_{\beta }^{s}},
	\end{eqnarray*}%
	where we used \eqref{help1}. Lemma \ref{Key-lemma5} gives%
	\begin{eqnarray*}
		\big\|\Pi _{2}(f,f,...,f)\big\|_{\dot{K}_{p,q}^{\alpha }F_{\beta }^{s}}
		&\lesssim &\Big\|\Big(\sum_{j=0}^{\infty }|2^{js}(Q_{j}f)^{|I_{2}|}(\Delta
		_{j}f)^{|I_{1}|}|^{\beta }\Big)^{\frac{1}{\beta }}\Big\|_{\dot{K}%
			_{p,q}^{\alpha }} \\
		&\lesssim &\big\|f\big\|_{\infty }^{m-1}\Big\|\Big(\sum_{j=0}^{\infty
		}|2^{js}\Delta _{j}f|^{\beta }\Big)^{\frac{1}{\beta }}\Big\|_{\dot{K}%
			_{p,q}^{\alpha }} \\
		&\lesssim &\big\|f\big\|_{\infty }^{m-1}\big\|f\big\|_{\dot{K}_{p,q}^{\alpha
			}F_{\beta }^{s}},
	\end{eqnarray*}%
	with the help of \eqref{help1}.
\end{proof}

\begin{rem}
	Theorem \ref{Multiplication1} in the case $m=2,p=q$ and $\alpha =0$ is
	contained in \cite{Ya86} and also in \cite{Si87}. For $m>2,p=q$ and $\alpha
	=0$ see \cite[Remark 17]{Si89} and \cite[p. 291]{RS96}. We refer the reader
	to the monograph \cite{RS96} and the paper \cite{Jo} for further details,
	historical remarks and more references on multiplication in Besov and
	Triebel-Lizorkin spaces.
\end{rem}

\begin{defn}
	Let $\mu >0$. Let $L\in \mathbb{N}_{0}$, and let $0<\nu \leq 1$ such that $%
	\mu =L+\nu $. The spaces $Lip\mu $ is the collection of all $f\in C^{L,%
		\mathrm{loc}}(\mathbb{R})$ such that%
	\begin{equation*}
	f^{(l)}(0)=0,\quad l=0,1,2,...,L
	\end{equation*}%
	and%
	\begin{equation*}
	\sup_{t_{0},t_{1}\in \mathbb{R}}\frac{|f^{(L)}(t_{0})-f^{(L)}(t_{1})|}{%
		|t_{0}-t_{1}|^{\nu }}<\infty .
	\end{equation*}%
	Then we put 
	\begin{equation*}
	\big\|f\big\|_{Lip\mu }=\sum_{j=0}^{L-1}\sup_{t\in \mathbb{R}}\frac{%
		|f^{(j)}(t)|}{|t|^{\mu -j}}+\sup_{t_{0},t_{1}\in \mathbb{R}}\frac{%
		|f^{(L)}(t_{0})-f^{(L)}(t_{1})|}{|t_{0}-t_{1}|^{\nu }}.
	\end{equation*}
\end{defn}

\begin{rem}
	$\big\|\cdot \big\|_{Lip\mu }$ defines not a norm, but for simplicity we
	will use this notation, see \cite[p. 295]{RS96}. A typical example of a
	function belongs to $Lip\mu $ is $f(t)=|t|^{\mu },\mu >1$. Recall that $%
	Lip\mu $ is not monotone with respect to $\mu $.
\end{rem}

We follow the same notations as in \cite[Chapter 5]{RS96}.

\begin{defn}
	$\mathrm{(i)}$ For $f\in \mathcal{S}^{\prime }(\mathbb{R}^{n})$ we define a
	distribution $\bar{f}$ by%
	\begin{equation*}
	\bar{f}(\varphi )=\overline{f(\bar{\varphi})},\quad \varphi \in \mathcal{S}(%
	\mathbb{R}^{n}).
	\end{equation*}%
	$\mathrm{(ii)}$ The space of real-valued distributions $\mathbb{S}^{\prime }(%
	\mathbb{R}^{n})$ is defined to be%
	\begin{equation*}
	\mathbb{S}^{\prime }(\mathbb{R}^{n})=\{f\in \mathcal{S}^{\prime }(\mathbb{R}%
	^{n}):\bar{f}=f\}.
	\end{equation*}%
	$\mathrm{(iii)}$ Let $A$ be a complex-valued, quasi-normed distribution
	space such that $A\hookrightarrow \mathcal{S}^{\prime }(\mathbb{R}^{n})$.
	Then we define the real-valued part $\mathbb{A}$ of $A$ to be the
	restriction of $A$ to $\mathbb{S}^{\prime }(\mathbb{R}^{n})$ equipped with
	the same quasi-norm as $A$.
\end{defn}

Now we are in position to state the first result of this section.

\begin{thm}
	\label{Key-theorem1}Let $0<p,q<\infty ,0<\beta \leq \infty ,\mu >1,\alpha
	\geq 0,s\geq \frac{n}{p}-\frac{n}{q}$ and 
	\begin{equation*}
	0<s<\frac{n}{p}+\alpha .
	\end{equation*}%
	We put 
	\begin{equation*}
	s_{\mu }=s-(\mu -1)\big(\frac{n}{p}+\alpha -s\big).
	\end{equation*}%
	Let $G\in Lip\mu $ and 
	\begin{equation}
	\max \Big(0,\frac{n}{p}+\alpha -n\Big)<s_{\mu }<\mu .  \label{assumption-mu}
	\end{equation}%
	Then 
	\begin{equation*}
	\big\|G(f)\big\|_{\dot{K}_{p,q}^{\alpha }F_{\beta }^{s_{\mu }}}\lesssim %
	\big\|G\big\|_{Lip\mu }\big\|f\big\|_{\dot{K}_{p,q}^{\alpha }F_{\infty
		}^{s}}^{\mu }
	\end{equation*}%
	holds for any $f\in \mathbb{\dot{K}}_{p,q}^{\alpha }\mathbb{F}_{\infty }^{s}$.
\end{thm}

\begin{proof}
	We will do the proof in three steps.
	
	\textbf{Step 1.}\textit{\ Preparation. }Consider the partition of the unity $%
	\{\mathcal{F}\varphi _{j}\}_{j\in \mathbb{N}_{0}}$. Let $f\in \dot{K}%
	_{p,q}^{\alpha }F_{\infty }^{s}$. We set%
	\begin{equation*}
	\frac{1}{b}=\frac{\frac{n}{p}+\alpha -s}{n+\alpha p}\quad \text{and}\quad
	\alpha _{1}=\frac{\alpha p}{b}.
	\end{equation*}%
	Then 
	\begin{equation}
	\max (1,p)<b<\infty \text{\quad and\quad }-\frac{n}{b}<\alpha _{1}<\min \big(%
	\alpha ,n-\frac{n}{b}\big).  \label{new-cond}
	\end{equation}%
	Hence 
	\begin{equation}
	\dot{K}_{p,q}^{\alpha }F_{\infty }^{s}\hookrightarrow \dot{K}_{b,r}^{\alpha
		_{1}},\quad \max (1,q,\mu )<r.  \label{embtoherz}
	\end{equation}%
	Since $G\in Lip\mu $, $\frac{b}{\mu }>1$ and $\alpha _{1}\mu <n-\frac{n\mu }{%
		b}$, we have%
	\begin{equation*}
	G(f)\in \dot{K}_{\frac{b}{\mu },\frac{r}{\mu }}^{\alpha _{1}\mu
	}\hookrightarrow \mathcal{S}^{\prime }(\mathbb{R}^{n})
	\end{equation*}%
	and so we can interpret $G$ as a mapping of a subspace of $\mathcal{S}%
	^{\prime }(\mathbb{R}^{n})$ into $\mathcal{S}^{\prime }(\mathbb{R}^{n})$. In
	addition%
	\begin{equation}
	f=\sum_{j=0}^{\infty }\varphi _{j}\ast f,\quad \text{in}\quad \dot{K}_{\mu
		,r}^{\alpha _{1}-\frac{n}{\mu }+\frac{n}{b}}.  \label{converges}
	\end{equation}%
	Indeed, let 
	\begin{equation*}
	\varrho _{k}=\sum\limits_{j=0}^{k}\varphi _{j}\ast f,\quad k\in \mathbb{N}%
	_{0}.
	\end{equation*}%
	Obviously $\{\varrho _{k}\}$ converges to $f$ in $\mathcal{S}^{\prime }(%
	\mathbb{R}^{n})$ and by the embedding \eqref{embtoherz}\ we derive that $%
	\{\varrho _{k}\}\subset \dot{K}_{b,r}^{\alpha _{1}}$. Furthermore, $%
	\{\varrho _{k}\}$ is a Cauchy sequences in $\dot{K}_{b,r}^{\alpha _{1}}$ and
	hence it converges to $g\in \dot{K}_{b,r}^{\alpha _{1}}$. Let us prove that $%
	f=g$ a.e. Let $\varphi \in \mathcal{D}(\mathbb{R}^{n})$. We write%
	\begin{equation*}
	\langle f-g,\varphi \rangle =\langle f-\varrho _{N},\varphi \rangle +\langle
	g-\varrho _{N},\varphi \rangle ,\quad N\in \mathbb{N}_{0}.
	\end{equation*}%
	Here $\langle \cdot ,\cdot \rangle $ denotes the duality bracket between $%
	\mathcal{D}^{\prime }(\mathbb{R}^{n})$ and $\mathcal{D}(\mathbb{R}^{n})$.
	Clearly, the first term tends to zero as $N\rightarrow \infty $, while by H%
	\"{o}lder's\ inequality there exists a constant $C>0$ independent of $N$
	such that 
	\begin{equation*}
	|\langle g-\varrho _{N},\varphi \rangle |\leq C\big\|g-\varrho _{N}\big\|_{%
		\dot{K}_{b,r}^{\alpha _{1}}},
	\end{equation*}%
	which tends to zero as $N\rightarrow \infty $. Therefore $f=g$ almost
	everywhere. Consequently, $f=\sum_{j=0}^{\infty }\varphi _{j}\ast f\ $in $%
	\dot{K}_{b,r}^{\alpha _{1}}$. Finally, \eqref{converges}, follows by the
	embedding $\dot{K}_{b,r}^{\alpha _{1}}\hookrightarrow \dot{K}_{\mu
		,r}^{\alpha _{1}-\frac{n}{\mu }+\frac{n}{b}}$. We have also%
	\begin{equation*}
	f=\sum_{j=0}^{\infty }\varphi _{j}\ast f,\quad \text{in}\quad L_{\mathrm{loc}%
	}^{\mu },
	\end{equation*}%
	because of $\dot{K}_{\mu ,r}^{\alpha _{1}-\frac{n}{\mu }+\frac{n}{b}%
	}\hookrightarrow L_{\mathrm{loc}}^{\mu }$. Indeed, let $B(0,2^{M})\subset 
	\mathbb{R}^{n},M\in \mathbb{Z}$. H\"{o}lder's inequality and the fact that $%
	\alpha _{1}-\frac{n}{\mu }+\frac{n}{b}=\frac{n}{p}+\alpha -\frac{n}{\mu }%
	-s<0 $, see \eqref{assumption-mu}, give%
	\begin{eqnarray*}
		\big\|f\big\|_{L^{\mu }(B(0,2^{M}))}^{\mu } &=&\sum_{i=-\infty }^{M}\big\|%
		f\chi _{R_{i}}\big\|_{\mu }^{\mu } \\
		&=&\sum_{i=-\infty }^{M}2^{-i(\alpha _{1}-\frac{n}{\mu }+\frac{n}{b})\mu
		}2^{i(\alpha _{1}-\frac{n}{\mu }+\frac{n}{b})\mu }\big\|f\chi _{R_{i}}\big\|%
		_{\mu }^{\mu } \\
		&\leq &C(M)\Big(\sum_{i=-\infty }^{M}2^{i(\alpha _{1}-\frac{n}{\mu }+\frac{n%
			}{b})r}\big\|f\chi _{i}\big\|_{\mu }^{r}\Big)^{\frac{\mu }{r}} \\
		&\lesssim &\big\|f\big\|_{\dot{K}_{\mu ,r}^{\alpha _{1}-\frac{n}{\mu }+\frac{%
					n}{b}}}^{\mu }.
	\end{eqnarray*}%
	We put $\mu =L+\nu $, where $0<\nu \leq 1$. The function $G$ has the Taylor
	expansion%
	\begin{equation*}
	G(t)=\sum_{l=0}^{L-1}\frac{G^{(l)}(z)}{l!}(t-z)^{l}+R(t,z),\quad t,z\in 
	\mathbb{R},
	\end{equation*}%
	where%
	\begin{equation*}
	R(t,z)=\frac{1}{L!}\int_{z}^{t}(t-y)^{L-1}G^{(L)}(y)dy.
	\end{equation*}%
	Since $f\in \dot{K}_{p,q}^{\alpha }F_{\infty }^{s}$ and $s>\max(0,\frac{n}{p}+\alpha-n)$
	there exists a set $A$ of Lebesgue-measure zero such that $|f(x)|<\infty $\
	for all $x\in \mathbb{R}^{n}\backslash A$. We can we suppose that $%
	|f(x)|<\infty $\ for all $x\in \mathbb{R}^{n}$. Therefore 
	\begin{eqnarray*}
		G(f(y)) &=&\sum_{l=0}^{L-1}\frac{1}{l!}%
		\sum_{j=0}^{l}(-1)^{l-j}C_{j}^{l}f^{j}(y)(\psi _{k}\ast
		f(x))^{l-j}G^{(l)}(\psi _{k}\ast f(x)) \\
		&&+R_{k}(f(y),\psi _{k}\ast f(x)),
	\end{eqnarray*}%
	where, $x,y\in \mathbb{R}^{n},$ 
	\begin{equation*}
	\psi _{k}\ast f=\sum_{i=0}^{k}\varphi _{i}\ast f,\quad k\in \mathbb{N}_{0}
	\end{equation*}%
	and 
	\begin{equation*}
	R_{k}(f(y),\psi _{k}\ast f(x))=\frac{1}{L!}\int_{\psi _{k}\ast
		f(x)}^{f(y)}(f(y)-h)^{L-1}G^{(L)}(h)dh.
	\end{equation*}%
	We put $K_{j,l}=(-1)^{l-j}C_{j}^{l}\frac{1}{l!}$, with $0\leq l\leq
	L-1,0\leq j\leq l$. Consequently%
	\begin{equation*}
	\varphi _{k}\ast G(f)(x)=\int_{\mathbb{R}^{n}}\varphi
	_{k}(x-y)G(f(y))dy=\sum_{l=0}^{L-1}\sum_{j=0}^{l}H_{k,1,j,l}(x)+H_{k,2}(x),
	\end{equation*}%
	where%
	\begin{eqnarray*}
		H_{k,1,j,l}(x) &=&K_{j,l}(\psi _{k}\ast f(x))^{l-j}G^{(l)}(\psi _{k}\ast
		f(x))\int_{\mathbb{R}^{n}}\varphi _{k}(x-y)f^{j}(y)dy \\
		&=&K_{j,l}(\psi _{k}\ast f(x))^{l-j}G^{(l)}(\psi _{k}\ast f(x))\varphi
		_{k}\ast f^{j}(x)
	\end{eqnarray*}%
	with $0\leq l\leq L-1,0\leq j\leq l$ and%
	\begin{equation*}
	H_{k,2}(x)=\frac{1}{L!}\int_{\mathbb{R}^{n}}\varphi _{k}(x-y)\int_{\psi
		_{k}\ast f(x)}^{f(y)}(f(y)-h)^{L-1}G^{(L)}(h)dhdy.
	\end{equation*}%
	We will estimate each term separately.
	
	\textbf{Step 2.}\textit{\ Estimate of} $H_{k,1,j,l}$. First assume that $%
	0<j\leq L-1$. Recall that%
	\begin{equation*}
	s_{i}=s-(i-1)\big(\frac{n}{p}+\alpha -s\big),\quad i>1
	\end{equation*}%
	and $s_{i}\leq s_{v},i\geq v>1$. Define%
	\begin{equation*}
	p_{1}=\frac{n+\alpha p}{(\mu -j)(\frac{n}{p}+\alpha -s)}
	\end{equation*}%
	and%
	\begin{equation}
	p_{2}=\frac{n+\alpha p}{\bar{s}-s_{j}+\frac{n}{p}+\alpha },  \label{p-two}
	\end{equation}%
	where%
	\begin{equation*}
	s_{\mu }<\bar{s}<\min (\mu ,s_{L}),
	\end{equation*}%
	with $0<j\leq l,0\leq l\leq L-1$. Since $s_{j}-\frac{n}{p}-\alpha =-j(\frac{n%
	}{p}+\alpha -s)<0$, \eqref{p-two} is well defined. We put $\frac{1}{\bar{p}}=%
	\frac{1}{p_{1}}+\frac{1}{p_{2}}$. Hence%
	\begin{equation*}
	\bar{s}-\frac{n+\alpha p}{\bar{p}}=s_{\mu }-\frac{n}{p}-\alpha ,\quad \bar{p}%
	<p<\min \big((\mu -j)p_{1},p_{2}\big).
	\end{equation*}%
	In addition%
	\begin{equation*}
	s-\frac{n}{p}-\alpha =\frac{-n}{(\mu -j)p_{1}}-\frac{\alpha p}{(\mu -j)p_{1}}%
	\quad \text{and}\quad s_{j}-\frac{n}{p}-\alpha =\bar{s}-\frac{n}{p_{2}}-%
	\frac{\alpha p}{p_{2}}.
	\end{equation*}%
	These choices guarantee the Sobolev embeddings%
	\begin{equation}
	\dot{K}_{p,q}^{\alpha }F_{p,\beta }^{s}\hookrightarrow \dot{K}_{(\mu
		-j)p_{1},\infty }^{\frac{\alpha p}{(\mu -j)p_{1}}}F_{1}^{0},\quad \dot{K}_{%
		\bar{p},q}^{\frac{\alpha p}{\bar{p}}}F_{\infty }^{\bar{s}}\hookrightarrow 
	\dot{K}_{p,q}^{\alpha }F_{\beta }^{s_{\mu }}  \label{Sobolev-emb1}
	\end{equation}%
	and 
	\begin{equation}
	\dot{K}_{p,q}^{\alpha }F_{p,\beta }^{s}\hookrightarrow \dot{K}_{p,q}^{\alpha
	}F_{r}^{s_{j}}\hookrightarrow \dot{K}_{p_{2},q}^{\frac{\alpha p}{p_{2}}%
	}F_{r}^{\bar{s}},\quad 0<r\leq \infty ,  \label{Sobolev-emb2}
	\end{equation}%
	see Theorem \ref{embeddings3}. We will prove that%
	\begin{equation*}
	\Big\|\sup_{k\in \mathbb{N}_{0}}2^{k\bar{s}}\big|H_{k,1,j,l}+H_{k,2}\big|%
	\Big\|_{\dot{K}_{\bar{p},q}^{\frac{\alpha p}{\bar{p}}}}\lesssim \big\|f\big\|%
	_{\dot{K}_{p,q}^{\alpha }F_{p,\beta }^{s}}^{\mu }.
	\end{equation*}%
	By H\"{o}lder's inequality and the fact that%
	\begin{equation}
	|G^{(l)}(t)|\leq \big\|G\big\|_{Lip\mu }|t|^{\mu -l},\quad t\in \mathbb{R}%
	,l=0,...,L-1  \label{G-properties}
	\end{equation}%
	we obtain\ that%
	\begin{eqnarray*}
		2^{k\bar{s}}\big\|H_{k,1,j,l}\big\|_{\dot{K}_{\bar{p},q}^{\frac{\alpha p}{%
					\bar{p}}}} &\lesssim &\big\||\psi _{k}\ast f|^{l-j}G^{(l)}(\psi _{k}\ast f)%
		\big\|_{\dot{K}_{p_{1},\infty }^{\frac{\alpha p}{p_{1}}}}2^{k\bar{s}}\big\|%
		\varphi _{k}\ast f^{j}\big\|_{\dot{K}_{p_{2},q}^{\frac{\alpha p}{p_{2}}}} \\
		&\lesssim &\big\|G\big\|_{Lip\mu }\big\|\psi _{k}\ast f\big\|_{\dot{K}_{(\mu
				-j)p_{1},\infty }^{\frac{\alpha p}{(\mu -j)p_{1}}}}^{\mu -j}2^{k\bar{s}}%
		\big\|\varphi _{k}\ast f^{j}\big\|_{\dot{K}_{p_{2},q}^{\frac{\alpha p}{p_{2}}%
		}} \\
		&\lesssim &\big\|G\big\|_{Lip\mu }\big\|f\big\|_{\dot{K}_{(\mu
				-j)p_{1},\infty }^{\frac{\alpha p}{(\mu -j)p_{1}}}F_{1}^{0}}^{\mu -j}\big\|%
		f^{j}\big\|_{\dot{K}_{p_{2},q}^{\frac{\alpha p}{p_{2}}}F_{\infty }^{\bar{s}}}
		\\
		&\lesssim &\big\|G\big\|_{Lip\mu }\big\|f\big\|_{\dot{K}_{p,q}^{\alpha
			}F_{p,\beta }^{s}}^{\mu -j}\big\|f^{j}\big\|_{\dot{K}_{p,q}^{\alpha
			}F_{r}^{s_{j}}} \\
		&\lesssim &\big\|G\big\|_{Lip\mu }\big\|f\big\|_{\dot{K}_{p,q}^{\alpha
			}F_{p,\beta }^{s}}^{\mu }
	\end{eqnarray*}%
	for any $k\in \mathbb{N}_{0}$, where we have used the embeddings %
	\eqref{Sobolev-emb1} and \eqref{Sobolev-emb2}, and Theorem \ref%
	{Multiplication1} with the fact that%
	\begin{equation*}
	s>\max \Big(0,\frac{n}{p}+\alpha -\frac{n}{\mu }\Big),
	\end{equation*}%
	see \eqref{assumption-mu}. Now we estimate $H_{k,1,0,l}$. Let us recall some
	properties of our system $\{\mathcal{F}\varphi _{k}\}_{k\in \mathbb{N}_{0}}$%
	. It holds%
	\begin{equation*}
	\int_{\mathbb{R}^{n}}\varphi _{k}(y)dy=0\quad \text{and}\quad \int_{\mathbb{R%
		}^{n}}\varphi _{0}(y)dy=c\neq 0\quad k\in \mathbb{N}.
	\end{equation*}%
	Therefore we need only to estimate $H_{0,1,0,l},0\leq l\leq L-1$. We have,
	again by \eqref{G-properties}, 
	\begin{equation*}
	\big\|H_{0,1,0,l}\big\|_{\dot{K}_{\bar{p},q}^{\frac{\alpha p}{\bar{p}}%
	}}\lesssim \big\|G\big\|_{Lip\mu }\big\||\varphi _{0}\ast f|^{\mu }\big\|_{%
		\dot{K}_{\bar{p},q}^{\frac{\alpha p}{\bar{p}}}}\lesssim \big\|G\big\|%
	_{Lip\mu }\big\|f\big\|_{\dot{K}_{\bar{p}\mu ,q}^{\frac{\alpha p}{\bar{p}\mu 
		}}F_{\infty }^{0}}^{\mu }.
	\end{equation*}%
	Thanks to the embeddings%
	\begin{equation}
	\dot{K}_{p,q}^{\alpha }F_{\beta }^{s}\hookrightarrow \dot{K}_{\bar{p}\mu
		,q}^{\frac{\alpha p}{\bar{p}\mu }}F_{\infty }^{\frac{\bar{s}}{\mu }%
	}\hookrightarrow \dot{K}_{\bar{p}\mu ,q}^{\frac{\alpha p}{\bar{p}\mu }%
	}F_{\infty }^{0},  \label{Sobolev-emb3}
	\end{equation}%
	because of%
	\begin{equation*}
	\bar{s}-\frac{n+\alpha p}{\bar{p}}=\mu (s-\frac{n}{p}-\alpha )\quad \text{and%
	}\quad \bar{p}\mu >p,
	\end{equation*}%
	we obtain%
	\begin{equation*}
	\big\|H_{0,1,0,l}\big\|_{\dot{K}_{\bar{p},q}^{\frac{\alpha p}{\bar{p}}%
	}}\lesssim \big\|G\big\|_{Lip\mu }\big\|f\big\|_{\dot{K}_{p,q}^{\alpha
		}F_{\beta }^{s}}^{\mu }.
	\end{equation*}%
	\textbf{Step 3.}\textit{\ Estimate } of $H_{k,2}$. We have%
	\begin{eqnarray*}
		&&\int_{\psi _{k}\ast f(x)}^{f(y)}(f(y)-h)^{L-1}G^{(L)}(h)dh \\
		&=&G^{(L)}(\psi _{k}\ast f(x))\frac{(f(y)-\psi _{k}\ast f(x))^{L}}{L} \\
		&&+\int_{\psi _{k}\ast f(x)}^{f(y)}(f(y)-h)^{L-1}(G^{(L)}(h)-G^{(L)}(\psi
		_{k}\ast f(x)))dh \\
		&=&H_{k,2,1}(x,y)+H_{k,2,2}(x,y).
	\end{eqnarray*}%
	The estimation of $H_{k,2,1}$ can be obtained by the same arguments given in
	Step 2. We estimate $H_{k,2,2}$. Using the fact that%
	\begin{equation*}
	|G^{(L)}(t_{0})-G^{(L)}(t_{1})|\leq \big\|G\big\|_{Lip\mu
	}|t_{0}-t_{1}|^{\nu },\quad t_{0},t_{1}\in \mathbb{R},
	\end{equation*}%
	we obtain%
	\begin{equation*}
	|H_{k,2,2}(x,y)|\lesssim \big\|G\big\|_{Lip\mu }|\psi _{k}\ast
	f(x)-f(y)|^{\mu },\quad x,y\in \mathbb{R}^{n}.
	\end{equation*}%
	Obviously%
	\begin{equation*}
	|\psi _{k}\ast f(x)-f(y)|\leq |\psi _{k}\ast
	f(x)-f(x)|+|f(x)-f(y)|,
	\end{equation*}%
	which yields that%
	\begin{equation*}
	\int_{\mathbb{R}^{n}}|\varphi _{k}(x-y)||H_{k,2,2}(x,y)|dy\leq
	S_{k,1}(f)(x)+S_{k,2}(f)(x),
	\end{equation*}%
	where%
	\begin{eqnarray*}
		S_{k,1}(f)(x) &=&\big\|G\big\|_{Lip\mu }\int_{\mathbb{R}^{n}}|\varphi
		_{k}(x-y)||\psi _{k}\ast f(x)-f(x)|^{\mu }dy \\
		&\lesssim &\big\|G\big\|_{Lip\mu }|\psi _{k}\ast f(x)-f(x)|^{\mu }
	\end{eqnarray*}%
	and%
	\begin{equation*}
	S_{k,2}(f)(x)=\big\|G\big\|_{Lip\mu }\int_{\mathbb{R}^{n}}|\varphi
	_{k}(x-y)||f(x)-f(y)|^{\mu }dy.
	\end{equation*}%
	First we estimate $S_{k,1}(f)$. Observe that%
	\begin{equation*}
	f-\psi _{k}\ast f=\sum_{i=k+1}^{\infty }\varphi _{i}\ast f,\quad k\in 
	\mathbb{N}_{0}.
	\end{equation*}%
	Therefore%
	\begin{eqnarray*}
		\Big\|\sup_{k\in \mathbb{N}_{0}}\big(2^{k\bar{s}}S_{k,1}(f)\big)\Big\|_{\dot{%
				K}_{\bar{p},q}^{\frac{\alpha p}{\bar{p}}}} &\lesssim &\big\|G\big\|_{Lip\mu }%
		\Big\|\sup_{k\in \mathbb{N}_{0}}2^{k\frac{\bar{s}}{\mu }}\Big(%
		\sum_{i=k+1}^{\infty }|\varphi _{i}\ast f|\Big)\Big\|_{\dot{K}_{\bar{p}\mu
				,q}^{\frac{\alpha p}{\bar{p}\mu }}}^{\mu } \\
		&\lesssim &\big\|G\big\|_{Lip\mu }\Big\|\sup_{k\in \mathbb{N}_{0}}2^{k\frac{%
				\bar{s}}{\mu }}|\varphi _{k}\ast f|\Big\|_{\dot{K}_{\bar{p}\mu ,q}^{\frac{%
					\alpha p}{\bar{p}\mu }}}^{\mu } \\
		&\lesssim &\big\|G\big\|_{Lip\mu }\big\|f\big\|_{\dot{K}_{\bar{p}\mu ,q}^{%
				\frac{\alpha p}{\bar{p}\mu }}F_{\infty }^{\frac{\bar{s}}{\mu }}}^{\mu },
	\end{eqnarray*}%
	where we used Lemma \ref{lq-inequality}. We conclude our desired estimate by
	the embeddings \eqref{Sobolev-emb3}.
	
	Now we estimate $S_{k,2}(f)$. Since $\psi ,\varphi \in \mathcal{S}(\mathbb{R}%
	^{n})$, this yields%
	\begin{equation*}
	|\varphi _{k}(z)|\lesssim \eta _{2^{k},M}(z),\quad z\in \mathbb{R}^{n},
	\end{equation*}%
	where $M$ is an arbitrary positive real number and the implicit constant is
	independent of $z$ and $k\in \mathbb{N}_{0}$. By means of this inequality we
	find 
	\begin{eqnarray*}
		&&\int_{\mathbb{R}^{n}}|\varphi _{k}(-z)||f(x)-f(x+z)|^{\mu }dz \\
		&\lesssim &\int_{\bar{B}_{k}}|\varphi _{k}(-z)||f(x)-f(x+z)|^{\mu }dz \\
		&&+\sum_{l=0}^{\infty }\int_{\bar{B}_{k-l-1}\backslash B_{k-l}}|\varphi
		_{k}(-z)||f(x)-f(x+z)|^{\mu }dz \\
		&\lesssim &2^{kn}\sum_{l=0}^{\infty }2^{-lM}\int_{\bar{B}%
			_{k-l-1}}|f(x)-f(x+z)|^{\mu }dz \\
		&\lesssim &2^{kn}\sum_{l=0}^{\infty }2^{-lM}I_{k-l}^{\mu }(f)(x),
	\end{eqnarray*}%
	where the implicit constant is independent of $x$ and $k$. Let $d=\min (1,%
	\bar{p})$. Taking $M$ large enough such that $M-n-\bar{s}-1>0$ and using
	Lemma \ref{Key-lemma1}, we obtain 
	\begin{eqnarray*}
		&&\Big\|\sup_{k\in \mathbb{N}_{0}}2^{k\bar{s}}\big|S_{k,2}(f)\big|\Big\|_{%
			\dot{K}_{\bar{p},q}^{\frac{\alpha p}{\bar{p}}}}^{d} \\
		&\lesssim &\big\|G\big\|_{Lip\mu }^{d}\sum_{l=0}^{\infty }2^{-lMd}\Big\|%
		\sup_{k\in \mathbb{N}_{0}}\big(2^{k(n+\bar{s})}I_{k-l}^{\mu }(f)\big)\Big\|_{%
			\dot{K}_{\bar{p},q}^{\frac{\alpha p}{\bar{p}}}}^{d} \\
		&\lesssim &\big\|G\big\|_{Lip\mu }^{d}\sum_{l=0}^{\infty }2^{-l(M-n-\bar{s}%
			)d}\Big\|\sup_{i\geq -l}\big(2^{i(n+\bar{s})}I_{i}^{\mu }(f)\big)\Big\|_{%
			\dot{K}_{\bar{p},q}^{\frac{\alpha p}{\bar{p}}}}^{d} \\
		&\lesssim &\big\|G\big\|_{Lip\mu }^{d}\big\|f\big\|_{\dot{K}_{\bar{p}\mu
				,q}^{\frac{\alpha p}{\bar{p}\mu }}F_{\infty }^{\frac{\bar{s}}{\mu }}}^{d}.
	\end{eqnarray*}%
	Our desired estimate follows by the embedding \eqref{Sobolev-emb3}. The
	proof is complete.
\end{proof}

From Theorem \ref{Key-theorem1} and the fact that $G(t)=|f|^{\mu }\in Lip\mu
,\mu >1$, we immediately arrive at the following result.

\begin{cor}
	Under the hypotheses of Theorem \ref{Key-theorem1}, we have 
	\begin{equation*}
	\big\||f|^{\mu }\big\|_{\dot{K}_{p,q}^{\alpha }F_{\beta }^{s_{\mu }}}\leq c%
	\big\|G\big\|_{Lip\mu }\big\|f\big\|_{\dot{K}_{p,q}^{\alpha }F_{\infty
		}^{s}}^{\mu }
	\end{equation*}%
	holds for any $f\in \mathbb{\dot{K}}_{p,q}^{\alpha }\mathbb{F}_{\infty }^{s}$%
	.
\end{cor}

\begin{rem}
	The valued $s_{\mu }$ in Theorem \ref{Key-theorem1} is optimal. Indeed, we
	put%
	\begin{equation*}
	f_{\kappa }(x)=\theta (x)|x|^{\kappa },
	\end{equation*}%
	where $\kappa >0$ and $\theta $ is a smooth cut-off function with $\mathrm{%
		supp}\theta \subset \{x:|x|\leq \vartheta \}$, $\vartheta >0$ sufficiently
	small. As in \cite{Dr21.1} we can prove that $f_{\kappa }\in \dot{K}%
	_{p,q}^{\alpha }F_{\beta }^{s}$ if and only if $s<\frac{n}{p}+\alpha +\kappa 
	$. Let $G(x)=|x|^{\mu },\mu >1,x\in \mathbb{R}$. Then 
	\begin{equation*}
	G(f_{\kappa })\notin \dot{K}_{p,q}^{\alpha }F_{\beta }^{d},
	\end{equation*}%
	if $d\geq \frac{n}{p}+\alpha +\kappa \mu >s_{\mu }$.
\end{rem}

\begin{thm}
	\label{Key-theorem1 copy(1)}Let $0<p,q<\infty ,0<\beta \leq \infty ,\mu
	>1,\alpha \geq 0$ and 
	\begin{equation*}
	\max \Big(0,\frac{n}{p}+\alpha -n\Big)<s<\mu .
	\end{equation*}%
	Let $G\in Lip\mu $. Then 
	\begin{equation*}
	\big\|G(f)\big\|_{\dot{K}_{p,q}^{\alpha }F_{\beta }^{s}}\leq c\big\|G\big\|%
	_{Lip\mu }\big\|f\big\|_{\dot{K}_{p,q}^{\alpha }F_{\beta }^{s}}\big\|f\big\|%
	_{\infty }^{\mu -1}
	\end{equation*}%
	holds for any $f\in \mathbb{\dot{K}}_{p,q}^{\alpha }\mathbb{F}_{\beta
	}^{s}\cap \mathbb{L}^{\infty }$.
\end{thm}

\begin{proof}
	We employ the notation of Theorem \ref{Key-theorem1}. We will prove that%
	\begin{equation*}
	\Big\|\sup_{k\in \mathbb{N}_{0}}2^{ks}\big|H_{k,1,j,l}+H_{k,2}\big|\Big\|_{%
		\dot{K}_{p,q}^{\alpha }}\lesssim \big\|G\big\|_{Lip\mu }\big\|f\big\|_{\dot{K%
		}_{p,q}^{\alpha }F_{\beta }^{s}}\big\|f\big\|_{\infty }^{\mu -1}.
	\end{equation*}%
	Thanks to \eqref{G-properties} and Theorem \ \ref{Multiplication1 copy(1)} it follows%
	\begin{eqnarray*}
		2^{ks}\big\|H_{k,1,j,l}\big\|_{\dot{K}_{p,q}^{\alpha }} &\lesssim &\big\|%
		|\psi _{k}\ast f|^{l-j}G^{(l)}(|\psi _{k}\ast f|)\big\|_{\infty }2^{ks}\big\|%
		\varphi _{k}\ast f^{j}\big\|_{\dot{K}_{p,q}^{\alpha }} \\
		&\lesssim &\big\|G\big\|_{Lip\mu }\big\|\psi _{k}\ast f\big\|_{\infty }^{\mu
			-j}2^{ks}\big\|\varphi _{k}\ast f^{j}\big\|_{\dot{K}_{p,q}^{\alpha }} \\
		&\lesssim &\big\|G\big\|_{Lip\mu }\big\|f\big\|_{\infty }^{\mu -j}\big\|f^{j}%
		\big\|_{\dot{K}_{p,q}^{\alpha }F_{\infty }^{s}} \\
		&\lesssim &\big\|G\big\|_{Lip\mu }\big\|f\big\|_{\infty }^{\mu -1}\big\|f%
		\big\|_{\dot{K}_{p,q}^{\alpha }F_{\infty }^{s}},
	\end{eqnarray*}%
	where we used $\big\|\psi _{k}\ast f\big\|_{\infty }\lesssim \big\|f\big\|%
	_{\infty }$, by Young's inequality. Now 
	\begin{eqnarray*}
		\big\|H_{0,1,0,l}\big\|_{\dot{K}_{p,q}^{\alpha }} &\lesssim &\big\|G\big\|%
		_{Lip\mu }\big\||\varphi _{0}\ast f|^{\mu }\big\|_{\dot{K}_{p,q}^{\alpha }}
		\\
		&\lesssim &\big\|G\big\|_{Lip\mu }\big\|f\big\|_{\infty }^{\mu -1}\big\|%
		\varphi _{0}\ast f\big\|_{\dot{K}_{p,q}^{\alpha }} \\
		&\lesssim &\big\|G\big\|_{Lip\mu }\big\|f\big\|_{\infty }^{\mu -1}\big\|f%
		\big\|_{\dot{K}_{p,q}^{\alpha }F_{\infty }^{s}}.
	\end{eqnarray*}%
	Observe that%
	\begin{equation*}
	S_{k,1}(f)(x)\lesssim \big\|G\big\|_{Lip\mu }|\psi _{k}\ast f(x)-f(x)|^{\mu
	}\lesssim \big\|G\big\|_{Lip\mu }\big\|f\big\|_{\infty }^{\mu -1}|\psi
	_{k}\ast f(x)-f(x)|.
	\end{equation*}%
	Then%
	\begin{eqnarray*}
		\Big\|\sup_{k\in \mathbb{N}_{0}}\big(2^{ks}S_{k,1}(f)\big)\Big\|_{\dot{K}%
			_{p,q}^{\alpha }} &\lesssim &\big\|G\big\|_{Lip\mu }\big\|f\big\|_{\infty
		}^{\mu -1}\Big\|\sup_{k\in \mathbb{N}_{0}}2^{ks}\Big(\sum_{i=k+1}^{\infty
		}|\varphi _{i}\ast f|\Big)\Big\|_{\dot{K}_{p,q}^{\alpha }} \\
		&\lesssim &\big\|G\big\|_{Lip\mu }\big\|f\big\|_{\infty }^{\mu -1}\Big\|%
		\sup_{k\in \mathbb{N}_{0}}2^{ks}|\varphi _{k}\ast f|\Big\|_{\dot{K}%
			_{p,q}^{\alpha }} \\
		&\lesssim &\big\|G\big\|_{Lip\mu }\big\|f\big\|_{\infty }^{\mu -1}\big\|f%
		\big\|_{\dot{K}_{p,q}^{\alpha }F_{\infty }^{s}},
	\end{eqnarray*}%
	by Lemma \ref{lq-inequality}. Using Lemma \ref{Key-lemma1}, we obtain%
	\begin{eqnarray*}
		&&\Big\|\sup_{k\in \mathbb{N}_{0}}2^{ks}\big|S_{k,2}(f)\big|\Big\|_{\dot{K}%
			_{p,q}^{\alpha }}^{d} \\
		&\lesssim &\big\|G\big\|_{Lip\mu }^{d}\sum_{l=0}^{\infty }2^{-lMd}\Big\|%
		\sup_{k\in \mathbb{N}_{0}}\big(2^{k(n+s)}I_{k-l}^{\mu }(f)\big)\Big\|_{\dot{K%
			}_{p,q}^{\alpha }}^{d} \\
		&\lesssim &\big\|G\big\|_{Lip\mu }^{d}\sum_{l=0}^{\infty }2^{-l(M-n-\bar{s}%
			)d}\Big\|\sup_{i\geq -l}\big(2^{i(n+s)}I_{i}^{\mu }(f)\big)\Big\|_{\dot{K}%
			_{p,q}^{\alpha }}^{d} \\
		&\lesssim &\big\|G\big\|_{Lip\mu }^{d}\big\|f\big\|_{\dot{K}_{p\mu ,q\mu }^{%
				\frac{\alpha }{\mu }}F_{\infty }^{\frac{s}{\mu }}}^{d\mu }.
	\end{eqnarray*}%
	The desired estimate follows by the fact that%
	\begin{equation*}
	\big\|f\big\|_{\dot{K}_{p\mu ,q\mu }^{\frac{\alpha }{\mu }}F_{\infty }^{%
			\frac{s}{\mu }}}^{\mu }\lesssim \big\|f\big\|_{\infty }^{\mu -1}\big\|f\big\|%
	_{\dot{K}_{p,q}^{\alpha }F_{\beta }^{s}}.
	\end{equation*}%
	The proof is completed.
\end{proof}

Now we present some limit case.

\begin{thm}
	\label{Key-theorem2}Let $0<p,q<\infty ,\alpha \geq 0,\mu \geq \frac{\frac{n}{%
			p}+\alpha }{\frac{n}{q}+\alpha +1}$\ and 
	\begin{equation}
	\max \big(1,\frac{n}{p}+\alpha -n\big)<\mu <\frac{n}{p}+\alpha .
	\label{new1}
	\end{equation}%
	Let $G\in Lip\mu $ and%
	\begin{equation}
	s=1+\frac{\mu -1}{\mu }\big(\frac{n}{p}+\alpha \big).  \label{new2}
	\end{equation}%
	Then 
	\begin{equation*}
	\big\|G(f)\big\|_{\dot{K}_{p,q}^{\alpha }F_{\infty }^{\mu }}\leq c\big\|G%
	\big\|_{Lip\mu }\big\|f\big\|_{\dot{K}_{p,q}^{\alpha }F_{\infty }^{s}}^{\mu }
	\end{equation*}%
	holds for any $f\in \mathbb{\dot{K}}_{p,q}^{\alpha }\mathbb{F}_{\infty }^{s}$%
	.
\end{thm}

\begin{proof}
	We employ the notation of the proof of Theorem \ref{Key-theorem1}. From %
	\eqref{new1} and \eqref{new2}, we obtain $\mu <s<\frac{n}{p}+\alpha $. With
	the help of \eqref{new1} we get \eqref{new-cond}, $\frac{b}{\mu }>1$ and $%
	\alpha _{1}\mu <n-\frac{n\mu }{b}$. Consequently the embedding %
	\eqref{embtoherz} holds. We have $s_{\mu }=\mu $ and we will take $\bar{s}%
	=s_{\mu }$ and $\bar{p}=p$. The proof is very similar as in Theorem \ref%
	{Key-theorem1}, but here we use Lemma \ref{Key-lemma2} instead of Lemma \ref%
	{Key-lemma1}.
\end{proof}

From Theorem \ref{Key-theorem2} and the fact that $G(t)=|f|^{\mu }\in Lip\mu
,\mu >1$, we get the following result:

\begin{cor}
	Under the hypotheses of Theorem \ref{Key-theorem2}, we have 
	\begin{equation*}
	\big\||f|^{\mu }\big\|_{\dot{K}_{p,q}^{\alpha }F_{\infty }^{\mu }}\leq c%
	\big\|G\big\|_{Lip\mu }\big\|f\big\|_{\dot{K}_{p,q}^{\alpha }F_{\infty
		}^{s}}^{\mu }
	\end{equation*}%
	holds for any $f\in \mathbb{\dot{K}}_{p,q}^{\alpha }\mathbb{F}_{\infty }^{s}$%
	.
\end{cor}

\begin{rem}
	Corresponding statements to Theorems \ref{Key-theorem1}, \ref{Key-theorem1
		copy(1)} and \ref{Key-theorem2} were proved in \cite{RS96} and \cite[Theorem
	6]{Si87} with $\alpha =0$, see also \cite{Ru86}.
\end{rem}

\section{Semilinear parabolic equations in Herz-Triebel-Lizorkin spaces}

\subsection{Heat kernel estimates}

Let $t>0,x\in \mathbb{R}^{n}$\ and\ $f\in \mathcal{S}^{\prime }\mathcal{(}%
\mathbb{R}^{n})$. We put 
\begin{equation*}
e^{t\Delta }f(x)=\mathcal{F}^{-1}(\exp (-t|\xi |^{2})\mathcal{F}f)(x).
\end{equation*}%
Recall that 
\begin{equation*}
g(x)=\mathcal{F}^{-1}(\exp (-t|\xi |^{2}))(x)=(4\pi t)^{-\frac{n}{2}}\exp
(-4t^{-1}|x|^{2}),\quad x\in \mathbb{R}^{n}.
\end{equation*}%
We will give some key estimates of heat kernel $e^{t\Delta }$\ needed in the
proofs of the main statements. First, we estimate the heat kernel $%
e^{t\Delta }$ in Herz-type Triebel-Lizorkin spaces. We follows the arguments
of \cite{FH17} and \cite{Triebel14}. We need the so called molecular and
wavelet characterizations of Herz-type Triebel-Lizorkin spaces.

\begin{defn}
	Let $K,L\in \mathbb{N}_{0}$ and $M>0$. A $K$-times continuous differentiable
	function $\mu $ is called a $[K,L,M]$-molecule concentrated in $Q_{j,m}$ if
	for some $j\in \mathbb{N}_{0}$ and $m\in \mathbb{Z}^{n}$%
	\begin{equation*}
	|D^{\gamma }\mu (x)|\leq 2^{|\gamma |j}(1+2^{j}|x-2^{-j}m|)^{-M},\quad 0\leq
	|\gamma |\leq K
	\end{equation*}%
	and%
	\begin{equation*}
	\int_{\mathbb{R}^{n}}x^{\gamma }\mu (x)dx=0\quad \text{if}\quad 0\leq |\gamma
	|<L,j\in \mathbb{N}.
	\end{equation*}
\end{defn}

Notice that for $L=0$ or $j=0$ there are no moment conditions on $\mu $. If $%
\mu $ is a molecule concentrated in $Q_{j,m}$, then it is denoted $\mu
_{j,m}.$

We introduce the sequence spaces associated with the function spaces $\dot{K}%
_{p,q}^{\alpha }F_{\beta }^{s}$. Let $\alpha ,s\in \mathbb{R},0<p,q<\infty $%
\ and $0<\beta \leq \infty $. We set%
\begin{equation*}
\dot{K}_{p,q}^{\alpha }f_{\beta }^{s}=\{\lambda =\{\lambda _{j,m}\}_{j\in 
	\mathbb{N}_{0},m\in \mathbb{Z}^{n}}\subset \mathbb{C}:\big\|\lambda \big\|_{%
	\dot{K}_{p,q}^{\alpha }f_{\beta }^{s}}<\infty \},
\end{equation*}%
where%
\begin{equation*}
\big\|\lambda \big\|_{\dot{K}_{p,q}^{\alpha }f_{\beta }^{s}}=\Big\|\Big(%
\sum_{j=0}^{\infty }\sum_{m\in \mathbb{Z}^{n}}2^{js\beta }|\lambda
_{j,m}|^{\beta }\chi _{j,m}\Big)^{1/\beta }\Big\|_{\dot{K}_{p,q}^{\alpha }}.
\end{equation*}%
Now we come to the molecule decomposition theorem for $\dot{K}_{p,q}^{\alpha
}F_{\beta }^{s}$ spaces. For the proof, see \cite{Drihem2.13} and \cite{Xu13}%
.

\begin{thm}
	\label{molecule}\textit{Let }$s\in \mathbb{R},0<p,q<\infty ,0<\beta \leq
	\infty $\textit{\ and }$\alpha >-\frac{n}{p}$\textit{. }Furthermore, let $%
	K,L\in \mathbb{N}_{0}$ and let $M>0$ with%
	\begin{equation*}
	L>\sigma _{p,\beta }-s,\text{\quad }K>s\text{ and }M\text{ large enough}.
	\end{equation*}%
	If $a_{j,m}$ are $\left[ K,L,M\right] $-molecules concentrated in $Q_{j,m}$
	and 
	\begin{equation*}
	\lambda =\{\lambda _{j,m}\}_{j\in \mathbb{N}_{0},m\in \mathbb{Z%
		}^{n}}\in \dot{K}_{p,q}^{\alpha }f_{\beta }^{s},
	\end{equation*}%
	then the sum%
	\begin{equation}
	f=\sum_{j=0}^{\infty }\sum_{m\in \mathbb{Z}^{n}}\lambda _{j,m}a_{j,m}
	\label{new-rep}
	\end{equation}%
	converges in $\mathcal{S}^{\prime }(\mathbb{R}^{n})$ and 
	\begin{equation*}
	\big\|f\big\|_{\dot{K}_{p,q}^{\alpha }F_{\beta }^{s}}\lesssim \big\|\lambda %
	\big\|_{\dot{K}_{p,q}^{\alpha }f_{\beta }^{s}}.
	\end{equation*}
\end{thm}

Let $J\in \mathbb{N}$ and $\psi _{F},\psi _{M}\in C^{J}(\mathbb{R})$ be
real-valued compactly supported Daubechies wavelets with%
\begin{equation*}
\mathcal{F}\psi _{F}(0)=(2\pi )^{-\frac{1}{2}},\quad \int_{\mathbb{R}%
}x^{l}\psi _{M}(x)dx=0,\quad l\in \{0,...,J-1\}
\end{equation*}%
and%
\begin{equation*}
\big\|\psi _{F}\big\|_{2}=\big\|\psi _{M}\big\|_{2}=1.
\end{equation*}%
We have that%
\begin{equation*}
\{\psi _{F}(x-m),2^{\frac{j}{2}}\psi _{M}(2^{j}x-m)\}_{j\in \mathbb{N}%
	_{0},m\in \mathbb{Z}^{n}}
\end{equation*}%
is an orthonormal basis in $L^{2}(\mathbb{R})$. This orthonormal basis can
be generalized to the $\mathbb{R}^{n}$ by the usual multiresolution
procedure. Let%
\begin{equation*}
G=\{G_{1},...,G_{n}\}\in G^{0}=\{F,M\}^{n}
\end{equation*}%
which means that $G_{r}$ is either $F$ or $M$. Let%
\begin{equation*}
G=\{G_{1},...,G_{n}\}\in G^{j}=\{F,M\}^{n^{\ast }},\quad j\in \mathbb{N},
\end{equation*}%
where indicates that at least one of the components of $G$ must be an $M$.
Let%
\begin{equation*}
\Psi _{G,m}^{j}(x)=2^{j\frac{n}{2}}\prod_{r=1}^{n}\psi
_{G_{r}}(2^{j}x_{r}-m_{r}),\quad G\in G^{j},m\in \mathbb{Z}^{n},x\in \mathbb{%
	R}^{n},j\in \mathbb{N}_{0}.
\end{equation*}%
Then%
\begin{equation*}
\Psi =\{\Psi _{G,m}^{j}:\quad j\in \mathbb{N}_{0},G\in G^{j},m\in \mathbb{Z}%
^{n}\}
\end{equation*}%
is an orthonormal basis in $L^{2}(\mathbb{R}^{n})$.

Let $\alpha ,s\in \mathbb{R},0<p,q<\infty $\ and $0<\beta \leq \infty $. We
set%
\begin{equation*}
\dot{K}_{p,q}^{\alpha }\tilde{f}_{\beta }^{s}=\{\lambda =\{\lambda
_{j,m}^{G}\}_{j\in \mathbb{N}_{0},G\in G^{j},m\in \mathbb{Z}^{n}}\subset 
\mathbb{C}:\big\|\lambda \big\|_{\dot{K}_{p,q}^{\alpha }\tilde{f}_{\beta
	}^{s}}<\infty \},
\end{equation*}%
where%
\begin{equation*}
\big\|\lambda \big\|_{\dot{K}_{p,q}^{\alpha }\tilde{f}_{\beta }^{s}}=\Big\|%
\Big(\sum_{j=0}^{\infty }\sum_{G\in G^{j}}\sum_{m\in \mathbb{Z}%
	^{n}}2^{js\beta }|\lambda _{j,m}^{G}|^{\beta }\chi _{j,m}\Big)^{1/\beta }%
\Big\|_{\dot{K}_{p,q}^{\alpha }}.
\end{equation*}

\begin{thm}
	\label{wavelet}\textit{Let }$\alpha ,s\in \mathbb{R},0<p,q<\infty ,0<\beta
	\leq \infty $ and $\alpha >-\frac{n}{p}$. Let $\{\Psi _{G,m}^{j}\}$ be the
	wavelet system with $J>\max (\sigma _{p,\beta }-s,s)$. Let $f\in \mathcal{S}%
	^{\prime }\mathcal{(}\mathbb{R}^{n})$. Then $f\in \dot{K}_{p,q}^{\alpha
	}F_{\beta }^{s}$ if and only if%
	\begin{equation}
	f=\sum_{j=0}^{\infty }\sum_{G\in G^{j}}\sum_{m\in \mathbb{Z}^{n}}\lambda
	_{j,m}^{G}2^{-j\frac{n}{2}}\Psi _{G,m}^{j},\quad \lambda \in \dot{K}%
	_{p,q}^{\alpha }\tilde{f}_{\beta }^{s}  \label{representation}
	\end{equation}%
	with unconditional convergence in $\mathcal{S}^{\prime }\mathcal{(}\mathbb{R}%
	^{n})$ and in any space $\dot{K}_{p,q}^{\alpha }F_{\beta }^{\sigma }$ with $%
	\sigma <s$. The representation \eqref{representation} is unique. We have%
	\begin{equation*}
	\lambda _{j,m}^{G}=\lambda _{j,m}^{G}(f)=2^{j\frac{n}{2}}\langle f,\Psi
	_{G,m}^{j}\rangle
	\end{equation*}%
	and%
	\begin{equation*}
	I:\quad f\longmapsto \{\lambda _{j,m}^{G}(f)\}
	\end{equation*}%
	is an isomorphic map from $\dot{K}_{p,q}^{\alpha }F_{\beta }^{s}$ into $\dot{%
		K}_{p,q}^{\alpha }\tilde{f}_{\beta }^{s}$. In particular, it holds 
	\begin{equation*}
	\big\|f\big\|_{\dot{K}_{p,q}^{\alpha }F_{\beta }^{s}}\approx \big\|\lambda %
	\big\|_{\dot{K}_{p,q}^{\alpha }\tilde{f}_{\beta }^{s}}.
	\end{equation*}
\end{thm}

For the proof, see again, \cite{Drihem2.13} and \cite{Xu13}. To estimate the
heat kernel $e^{t\Delta }$ in Herz-type Triebel-Lizorkin spaces, we need the
following lemma.

\begin{lem}
	\label{Heat-kernel3.1}Let $s>0,\theta \geq 0,0<t<T,0<p,q<\infty ,0<\beta
	\leq \infty $ and $\alpha >-\frac{n}{p}$. We set%
	\begin{equation*}
	b_{G,m}^{j}(x,t)=2^{-j\frac{n}{2}}e^{t\Delta }\Psi _{G,m}^{j}(x).
	\end{equation*}%
	Then there exists $C>0$ such that the functions%
	\begin{equation}
	b_{G,m}^{j}(x,t)_{\theta }=C2^{j\theta }t^{\frac{\theta }{2}%
	}b_{G,m}^{j}(x,t),\quad j\in \mathbb{N}_{0},G\in G^{\ast },m\in \mathbb{Z}%
	^{n}  \label{molecule1}
	\end{equation}%
	$\left[ K,L,M\right] $-molecules for any fixed $t$ with $2^{j}t^{\frac{1}{2}%
	}\geq 1$, provided that $L\leq J,K\leq J,L+n-1<M<J+n-\theta $ and $\theta
	\leq J-L+1$. Assume that%
	\begin{equation*}
	J>\theta +\max (s,\sigma _{p,\beta }).
	\end{equation*}%
	Then, the numbers $K,L,M$ can be chosen such that for some $C>0$ and any $t$
	with $2^{j}t^{\frac{1}{2}}\geq 1$, such that \eqref{molecule1} are molecules
	for $\dot{K}_{p,q}^{\alpha }F_{\beta }^{s+\theta }$.
\end{lem}

\begin{proof}
	We use the arguments of \cite[Proposition 3.1]{FH17} and we need only to prove
	the second part of the Lemma. Let $L=\lfloor \sigma _{p,\beta }\rfloor +1$,
	which yields that $L>\sigma _{p,\beta }-s-\theta $. Since $J>\sigma
	_{p,\beta }$ it follows that $J\geq L$. Hence%
	\begin{equation*}
	\int_{\mathbb{R}^{n}}x^{\nu }b_{G,m}^{j}(x,t)_{\theta }dx=0,\quad \quad
	0\leq |\nu |<L,j\in \mathbb{N}.
	\end{equation*}%
	Let $M$ large enough be such that $\sigma _{p,\beta }+n<M<J+n-\theta $. Then 
	$M>L+n-1$ and $\theta <J-\sigma _{p,\beta }<J-L+1$. Regarding the
	derivatives of $b_{G,m}^{j}(x,t)_{\theta }$ we claim $s+\theta <K\leq J$.
\end{proof}

We present one of the main tools used in this section.

\begin{lem}
	\label{Heat-kernel3}Let $s>0,\theta \geq 0,0<t<T,1<p,q<\infty ,1<\beta \leq
	\infty $ and $-\frac{n}{p}<\alpha< n-\frac{n}{p}$. \textit{Then there exists a positive
		constant }$C(T)>0$\textit{\ independent of }$t$\textit{\ such that}%
	\begin{equation*}
	\big\|e^{t\Delta }f\big\|_{\dot{K}_{p,q}^{\alpha }F_{\beta }^{s+\theta
	}}\leq C(T)t^{-\frac{\theta }{2}}\big\|f\big\|_{\dot{K}_{p,q}^{\alpha
		}F_{\beta }^{s}}
	\end{equation*}%
	for any $f\in \dot{K}_{p,q}^{\alpha }F_{\beta }^{s}$.
\end{lem}

\begin{proof}
	Let $k\in \mathbb{N}$ be such that $2^{-2k}<\frac{t}{T}\leq 2^{-2(k-1)}$.
	From Theorem \ref{wavelet} we have $f=f_{1,k}+f_{2,k}$, with%
	\begin{equation*}
	f_{1,k}=\sum_{j=0}^{k-1}\sum_{G\in G^{j}}\sum_{m\in \mathbb{Z}^{n}}\lambda
	_{j,m}^{G}2^{-j\frac{n}{2}}\Psi _{G,m}^{j}
	\end{equation*}%
	and%
	\begin{equation*}
	f_{2,k}=\sum_{j=k}^{\infty }\sum_{G\in G^{j}}\sum_{m\in \mathbb{Z}%
		^{n}}\lambda _{j,m}^{G}2^{-j\frac{n}{2}}\Psi _{G,m}^{j},
	\end{equation*}%
	where $\lambda \in \dot{K}_{p,q}^{\alpha }\tilde{f}_{\beta }^{s}$.
	
	\textit{Estimate of} $f_{1,k}$. We claim that%
	\begin{equation}
	|e^{t\Delta }(\varphi _{j}\ast f_{1,k})(x)|\lesssim \mathcal{M}(\varphi
	_{j}\ast f_{1,k})(x),\quad x\in \mathbb{R}^{n},j\in \mathbb{N}_{0},
	\label{claim}
	\end{equation}%
	where the implicit constant is independent of $x,k,j$ and $t$. Using
	the estimate \eqref{claim} and Lemma \ref{Maximal-Inq} we obtain%
	\begin{eqnarray}
	\big\|e^{t\Delta }f_{1,k}\big\|_{\dot{K}_{p,q}^{\alpha }F_{\beta }^{s+\theta
	}} &=&\Big\|\Big(\sum_{j=0}^{\infty }2^{j(s+\theta )\beta }|\varphi _{j}\ast
	e^{t\Delta }f_{1,k}|^{\beta }\Big)^{1/\beta }\Big\|_{\dot{K}_{p,q}^{\alpha }}
	\notag \\
	&=&\Big\|\Big(\sum_{j=0}^{\infty }2^{j(s+\theta )\beta }|e^{t\Delta
	}(\varphi _{j}\ast f_{1,k})|^{\beta }\Big)^{1/\beta }\Big\|_{\dot{K}%
		_{p,q}^{\alpha }}  \notag \\
	&\lesssim &\Big\|\Big(\sum_{j=0}^{\infty }2^{j(s+\theta )\beta }|\mathcal{M}%
	(\varphi _{j}\ast f_{1,k})|^{\beta }\Big)^{1/\beta }\Big\|_{\dot{K}%
		_{p,q}^{\alpha }}  \notag \\
	&\lesssim &\Big\|\Big(\sum_{j=0}^{\infty }2^{j(s+\theta )\beta }|\varphi
	_{j}\ast f_{1,k}|^{\beta }\Big)^{1/\beta }\Big\|_{\dot{K}_{p,q}^{\alpha }}.
	\label{estimate-f1}
	\end{eqnarray}%
	In view of the definition of the spaces $\dot{K}_{p,q}^{\alpha }F_{\beta
	}^{s+\theta }$, \eqref{estimate-f1} is just $\big\|f_{1,k}\big\|_{\dot{K}%
		_{p,q}^{\alpha }F_{\beta }^{s+\theta }}$. Thanks to Theorem \ref{wavelet} we
	get%
	\begin{eqnarray}
	\big\|f_{1,k}\big\|_{\dot{K}_{p,q}^{\alpha }F_{\beta }^{s+\theta }}
	&\lesssim &\Big\|\Big(\sum_{j=0}^{k-1}\sum_{G\in G^{j}}\sum_{m\in \mathbb{Z}%
		^{n}}2^{j(s+\theta )\beta }|\lambda _{j,m}^{G}|^{\beta }\chi _{j,m}\Big)%
	^{1/\beta }\Big\|_{\dot{K}_{p,q}^{\alpha }}  \notag \\
	&\lesssim &2^{k\theta }\Big\|\Big(\sum_{j=0}^{k-1}\sum_{G\in
		G^{j}}\sum_{m\in \mathbb{Z}^{n}}2^{js\beta }|\lambda _{j,m}^{G}|^{\beta
	}\chi _{j,m}\Big)^{1/\beta }\Big\|_{\dot{K}_{p,q}^{\alpha }}  \notag \\
	&=&ct^{-\frac{\theta }{2}}\big\|\lambda \big\|_{\dot{K}_{p,q}^{\alpha }%
		\tilde{f}_{\beta }^{s}}  \notag \\
	&\lesssim &t^{-\frac{\theta }{2}}\big\|f\big\|_{\dot{K}_{p,q}^{\alpha
		}F_{\beta }^{s}}.  \label{estimate-f1.1}
	\end{eqnarray}%
	Substituting \eqref{estimate-f1.1} into \eqref{estimate-f1}, this gives the
	desired estimate. Now we prove our claim. Since $g\in \mathcal{S}\left( 
	\mathbb{R}^{n}\right) $, we have%
	\begin{equation*}
	|e^{t\Delta }(\varphi _{j}\ast f_{1,k})(x)|\lesssim \eta _{t^{-\frac{1}{2}%
		},m}\ast |\varphi _{j}\ast f_{1,k}|(x),\quad m>n,
	\end{equation*}%
	which can be estimated by 
	\begin{equation}
	c\eta _{t^{-\frac{1}{2}},m}\chi _{B(x,2t^{\frac{1}{2}})}\ast |\varphi
	_{j}\ast f_{1,k}|(x)+\eta _{t^{-\frac{1}{2}},m}\chi _{\mathbb{R}%
		^{n}\backslash B(x,2t^{\frac{1}{2}})}\ast |\varphi _{j}\ast f_{1,k}|(x).
	\label{estimate-f1.2.1}
	\end{equation}%
	Obviously, the first term of \eqref{estimate-f1.2.1} is bounded by $c%
	\mathcal{M}(\varphi _{j}\ast f_{1,k})(x)$. We have%
	\begin{eqnarray*}
		&&\eta _{t^{-\frac{1}{2}},m}\chi _{\mathbb{R}^{n}\backslash B(x,2t^{\frac{1}{%
					2}})}\ast |\varphi _{j}\ast f_{1,k}|(x) \\
		&=&\sum_{i=1}^{\infty }\eta _{t^{-\frac{1}{2}},m}\chi _{B(x,2^{i+1}t^{\frac{1}{%
					2}})\backslash B(x,2^{i}t^{\frac{1}{2}})}\ast |\varphi _{j}\ast f_{1,k}|(x)
		\\
		&\leq &\sum_{i=1}^{\infty }2^{-im}\eta _{t^{-\frac{1}{2}},m}\chi
		_{B(x,2^{i+1}t^{\frac{1}{2}})}\ast |\varphi _{j}\ast f_{1,k}|(x) \\
		&\lesssim &\mathcal{M}(\varphi _{j}\ast f_{1,k})(x)\sum_{i=1}^{\infty
		}2^{i(n-m)} \\
		&\lesssim &\mathcal{M}(\varphi _{j}\ast f_{1,k})(x).
	\end{eqnarray*}
	
	\textit{Estimate of} $f_{2,k}$. If $j\geq k$, then $2^{j}(\frac{t}{T})^{%
		\frac{1}{2}}>2^{j-k}\geq 1$, which yields that%
	\begin{eqnarray*}
		e^{t\Delta }f_{2,k} &=&\sum_{j=k}^{\infty }\sum_{G\in G^{j}}\sum_{m\in 
			\mathbb{Z}^{n}}2^{-j\theta }(\frac{t}{T})^{-\frac{\theta }{2}}\lambda
		_{j,m}^{G}2^{-j\frac{n}{2}}2^{j\theta }(\frac{t}{T})^{\frac{\theta }{2}%
		}e^{t\Delta }\Psi _{G,m}^{j} \\
		&=&\sum_{j=k}^{\infty }\sum_{G\in G^{j}}\sum_{m\in \mathbb{Z}^{n}}\mu
		_{j,m}^{G}b_{G,m}^{j}(x,t)_{\theta },
	\end{eqnarray*}%
	where%
	\begin{equation*}
	C\mu _{j,m}^{G}=2^{-j\theta }(\frac{t}{T})^{-\frac{\theta }{2}}\lambda
	_{j,m}^{G}\quad \text{and}\quad b_{G,m}^{j}(x,t)_{\theta }=2^{-j\frac{n}{2}%
	}2^{j\theta }(\frac{t}{T})^{\frac{\theta }{2}}\Psi _{G,m}^{j},
	\end{equation*}%
	and $C$ as in \eqref{molecule1}. Let%
	\begin{equation*}
	\mu ^{\ast }=\{2^{-j\theta }(\frac{t}{T})^{-\frac{\theta }{2}}\lambda
	_{j,m}^{G},j\in \mathbb{N}_{0},G\in G^{\ast },m\in \mathbb{Z}^{n}\}.
	\end{equation*}%
	Again, from Theorem \ref{wavelet} we obtain%
	\begin{eqnarray*}
		\big\|e^{t\Delta }f_{2,k}\big\|_{\dot{K}_{p,q}^{\alpha }F_{\beta }^{s+\theta
		}} &\lesssim &\big\|\mu ^{\ast }\big\|_{\dot{K}_{p,q}^{\alpha }\tilde{f}%
			_{\beta }^{s+\theta }} \\
		&=&ct^{-\frac{\theta }{2}}\Big\|\Big(\sum_{j=k}^{\infty }\sum_{G\in
			G^{j}}\sum_{m\in \mathbb{Z}^{n}}2^{js\beta }|\lambda _{j,m}^{G}|^{\beta
		}\chi _{j,m}\Big)^{1/\beta }\Big\|_{\dot{K}_{p,q}^{\alpha }} \\
		&=&ct^{-\frac{\theta }{2}}\big\|\lambda \big\|_{\dot{K}_{p,q}^{\alpha }%
			\tilde{f}_{\beta }^{s}} \\
		&\lesssim &t^{-\frac{\theta }{2}}\big\|f\big\|_{\dot{K}_{p,q}^{\alpha
			}F_{\beta }^{s}}
	\end{eqnarray*}%
	and this completes the proof.
\end{proof}

The following lemmas was proved in \cite{Dr-Herz-Heat}.

\begin{lem}
	\label{Heat-kernel2}\textit{Let }$\alpha _{1},\alpha _{2}\in \mathbb{R}%
	,0<t<\infty \mathit{\ }$\textit{and} $1<p,\kappa ,q,r<\infty $. \textit{We
		suppose that }$1<q\leq p<\infty $ and $-\frac{n}{p}<\alpha _{1}\leq \alpha
	_{2}<n-\frac{n}{q}$. \textit{Then there exists a positive constant }$C>0$%
	\textit{\ independent of }$t$\textit{\ such that}%
	\begin{equation*}
	\big\|e^{t\Delta }f\big\|_{\dot{K}_{p,r}^{\alpha _{1}}}\leq Ct^{-\frac{1}{2}(%
		\frac{n}{q}-\frac{n}{p}+\alpha _{2}-\alpha _{1})}\big\|f\big\|_{\dot{K}%
		_{q,\delta }^{\alpha _{2}}}
	\end{equation*}%
	for any $f\in \dot{K}_{q,\delta }^{\alpha _{2}}$, where%
	\begin{equation*}
	\delta =\left\{ 
	\begin{array}{ccc}
	r, & \text{if} & \alpha _{2}=\alpha _{1}, \\ 
	\kappa , & \text{if} & \alpha _{2}>\alpha _{1}.%
	\end{array}%
	\right.
	\end{equation*}
\end{lem}

\subsection{The results and their proofs.}

We look for mild solutions of  \eqref{equation} i.e.  for solutions of  integral
equation

\begin{equation}
u(t,x)=e^{t\Delta }u_{0}(x)+\int_{0}^{t}e^{(t-\tau )\Delta }G(u)(\tau
,x)d\tau .  \label{int-equa}
\end{equation}%
We set%
\begin{equation*}
F(u)(t,x)=\int_{0}^{t}e^{(t-\tau )\Delta }G(u)(\tau ,x)d\tau .
\end{equation*}%
We study Cauchy problem for semilinear parabolic equations \eqref{equation}
with initially data in Herz-type Triebel-Lizorkin spaces and will assume
that $G\ $belongs to $G\in Lip\mu $. We set%
\begin{equation*}
\bar{s}=\frac{n}{p}+\alpha -\frac{2}{\mu -1}\quad \text{and}\quad \vartheta =%
\frac{s-\bar{s}}{2}.
\end{equation*}

We now state the existence of mild solutions of \eqref{int-equa}.

\begin{thm}
	\label{Theorem1}Let $1<p,q<\infty ,1<\beta \leq \infty ,\mu >1,0 \leq\alpha< n-\frac{n}{p}
	,s\geq \frac{n}{p}-\frac{n}{q}$ and 
	\begin{equation*}
	0<s<\frac{n}{p}+\alpha .
	\end{equation*}%
	Let $G\in Lip\mu $ and 
	\begin{equation*}
	0<s_{\mu }<\mu .
	\end{equation*}%
	$\mathrm{(i)}$ For all initial data $u_{0}$ in $\dot{K}_{p,q}^{\alpha
	}F_{\beta }^{s}$ with $s>\bar{s}$, there exists a maximal solution $u$ to %
	\eqref{int-equa} in $C([0,T_{0}),\dot{K}_{p,q}^{\alpha }F_{\beta }^{s})$
	with $T_{0}\geq C\big\|u_{0}\big\|_{\dot{K}_{p,q}^{\alpha }F_{\beta }^{s}}^{-%
		\frac{1}{\vartheta }}$.$\newline
	\mathrm{(ii)}$ Let $\theta <2\vartheta (\mu -1)$ or $\theta =2\vartheta (\mu
	-1),s>1$ and $G\in Lips_{0}$\ with 
	\begin{equation*}
	s_{0}=\frac{\frac{n}{p}+\alpha }{\frac{n}{p}+\alpha -s+1}.
	\end{equation*}%
	We have 
	\begin{equation*}
	u-e^{t\Delta }u_{0}\in C([0,T_{0}),\dot{K}_{p,q}^{\alpha }F_{\beta
	}^{s+\theta }).
	\end{equation*}
\end{thm}

\begin{proof}
	We will do the proof into two steps. Our arguments are based on \cite{Ri98}.
	
	\textbf{Step 1.}\textit{\ }We prove part (i) of the theorem.
	
	\textbf{Substep 1.1.} In this step we prove the existence of a solution to %
	\eqref{int-equa}. Recall that%
	\begin{equation*}
	F(u)(t,x)=\int_{0}^{t}e^{(t-\tau )\Delta }G(u)(\tau ,x)d\tau \quad \text{and}%
	\quad \frac{1}{\tilde{p}}=\frac{1}{p}+\frac{\alpha -s}{n}.
	\end{equation*}%
	For simplicity, we consider the spaces%
	\begin{equation*}
	Y=C([0,T),\dot{K}_{p,q}^{\alpha }F_{\beta }^{s})\quad \text{and}\quad
	X=C([0,T),\dot{K}_{\tilde{p},q}^{0}).
	\end{equation*}%
	Further,  we consider the sequence of functions%
	\begin{equation}
	u^{0}=e^{t\Delta }u_{0}\quad \text{and}\quad u^{j+1}=u^{0}+F(u^{j}),\quad
	j\in \mathbb{N}.  \label{pointfixe}
	\end{equation}%
	From Lemma \ref{Heat-kernel2} and Sobolev embedding $\dot{K}_{p,q}^{\alpha
	}F_{\beta }^{s}\hookrightarrow \dot{K}_{\tilde{p},q}^{0}$, see Theorem \ref%
	{embeddings3}, we deduce that 
	\begin{equation}
	\big\|u^{0}\big\|_{\dot{K}_{\tilde{p},q}^{0}}\lesssim \big\|u_{0}\big\|_{%
		\dot{K}_{\tilde{p},q}^{0}}\lesssim \big\|u_{0}\big\|_{\dot{K}_{p,q}^{\alpha
		}F_{\beta }^{s}}.  \label{embYintoX1}
	\end{equation}%
	Let $u,v\in X$. Since, $\frac{\tilde{p}}{\mu }>1$, again, by Lemma \ref%
	{Heat-kernel2} we obtain%
	\begin{eqnarray}
	&&\big\|F(u)(t,\cdot )-F(v)(t,\cdot )\big\|_{\dot{K}_{\tilde{p},q}^{0}} 
	\notag \\
	&\leq &\int_{0}^{t}\big\|e^{(t-\tau )\Delta }(G(u)(\tau ,\cdot )-G(v)(\tau
	,\cdot ))\big\|_{\dot{K}_{\tilde{p},q}^{0}}d\tau  \notag \\
	&\leq &\int_{0}^{t}\big\|e^{(t-\tau )\Delta }(G(u)(\tau ,\cdot )-G(v)(\tau
	,\cdot ))\big\|_{\dot{K}_{\tilde{p},\frac{q}{\mu }}^{0}}d\tau  \notag \\
	&\leq &C\int_{0}^{t}(t-\tau )^{-\frac{n(\mu -1)}{2\tilde{p}}}\big\|G(u)(\tau
	,\cdot )-G(v)(\tau ,\cdot )\big\|_{\dot{K}_{\frac{\tilde{p}}{\mu },\frac{q}{%
				\mu }}^{0}}d\tau ,  \label{est-u-v}
	\end{eqnarray}%
	where the second estimate follows by the embedding $\dot{K}_{\tilde{p},\frac{%
			q}{\mu }}^{0}\hookrightarrow \dot{K}_{\tilde{p},q}^{0}$ and the positive
	constant $C$ is independent of $t$. Observe that%
	\begin{equation*}
	|G(u)(\tau ,\cdot )-G(v)(\tau ,\cdot )|\leq |u-v|(|u|^{\mu -1}+|v|^{\mu -1})
	\end{equation*}%
	and%
	\begin{equation*}
	\frac{\mu }{\tilde{p}}=\frac{1}{\tilde{p}}+\frac{\mu -1}{\tilde{p}},\quad 
	\frac{\mu }{q}=\frac{1}{q}+\frac{\mu -1}{q}.
	\end{equation*}%
	Therefore, by H\"{o}lder's inequality%
	\begin{eqnarray}
	&&\big\|G(u)(\tau ,\cdot )-G(v)(\tau ,\cdot )\big\|_{\dot{K}_{\frac{\tilde{p}%
			}{\mu },\frac{q}{\mu }}^{0}}  \label{est-u-v2} \\
	&\leq &\big\|u(\tau ,\cdot )-v(\tau ,\cdot )\big\|_{\dot{K}_{\tilde{p}%
			,q}^{0}}\Big(\big\|u(\tau ,\cdot )\big\|_{\dot{K}_{\tilde{p},q}^{0}}^{\mu
		-1}+\big\|v(\tau ,\cdot )\big\|_{\dot{K}_{\tilde{p},q}^{0}}^{\mu -1}\Big). 
	\notag
	\end{eqnarray}%
	Substituting \eqref{est-u-v2} into \eqref{est-u-v} and then using 
	\begin{equation*}
	\frac{n(\mu -1)}{2\tilde{p}}=	\frac{(\mu -1)}{2}(\frac{n}{p}+\alpha -s)=1-\frac{(\mu
		-1)(s-\bar{s})}{2},
	\end{equation*}%
	this gives%
	\begin{equation}
	\big\|F(u)-F(v)\big\|_{X}\leq CT^{\frac{(\mu -1)(s-\bar{s})}{2}}\big\|u-v%
	\big\|_{X}\Big(\big\|u\big\|_{X}^{\mu -1}+\big\|v\big\|_{X}^{\mu -1}\Big).
	\label{pointfixe2}
	\end{equation}%
	In view of \eqref{pointfixe}, \eqref{pointfixe2} and \eqref{embYintoX1}, we
	obtain%
	\begin{eqnarray*}
		\big\|u^{j+1}\big\|_{X} &\lesssim &\big\|u^{0}\big\|_{X}+\big\|F(u^{j})\big\|%
		_{X} \\
		&\leq &\big\|u_{0}\big\|_{\dot{K}_{p,q}^{\alpha }F_{\beta }^{s}}+CT^{\frac{%
				(\mu -1)(s-\bar{s})}{2}}\big\|u^{j}\big\|_{X}^{\mu }
	\end{eqnarray*}%
	and%
	\begin{equation*}
	\big\|u^{j+1}-u^{j}\big\|_{X}\leq CT^{\frac{(\mu -1)(s-\bar{s})}{2}}\big\|%
	u^{j}-u^{j-1}\big\|_{X}\Big(\big\|u^{j}\big\|_{X}^{\mu -1}+\big\|u^{j-1}%
	\big\|_{X}^{\mu -1}\Big).
	\end{equation*}%
	Let 
	\begin{equation*}
	\digamma =(\frac{1}{C})^{\frac{2}{(\mu -1)(s-\bar{s})}}\big(\mu ^{\frac{-1}{%
			\mu -1}}-\mu ^{\frac{-\mu }{\mu -1}}\big)^{\frac{2}{s-\bar{s}}}.
	\end{equation*}%
	As in \cite{Dr-Herz-Heat} and \cite{KY94}, the fixed point argument shows
	that if 
	\begin{equation}
	T<\digamma 2^{\frac{-2}{(\mu -1)(s-\bar{s})}}(1-\frac{1}{\mu })^{\mu -1}%
	\big\|u_{0}\big\|_{\dot{K}_{p,q}^{\alpha }F_{\beta }^{s}}^{\frac{-2}{s-\bar{s%
	}}},  \label{T-hyp}
	\end{equation}%
	then the sequence $\{u^{j}\}_{j}$ converges strongly in $X$ to a limit $u$
	which is a solution of the integral equation \eqref{int-equa}.
	
	\textbf{Substep 1.2.} In this step we prove that the solution of the
	integral equation \eqref{int-equa} belongs to $Y$. We employ the notation of
	Substep 1.1. We claim that%
	\begin{equation}
	\big\|u^{j+1}\big\|_{Y}\leq \big\|u_{0}\big\|_{\dot{K}_{p,q}^{\alpha
		}F_{\beta }^{s}}+CT^{\frac{(\mu -1)(s-\bar{s})}{2}}\big\|u^{j}\big\|%
	_{Y}^{\mu }.  \label{uj}
	\end{equation}%
	From \eqref{T-hyp} and \eqref{uj}, the sequence $\{u^{j}\}_{j}$ is bounded.
	Then we can extract a subsequence $\{u^{j_{i}}\}_{i}$ converges weakly to $%
	\tilde{u}\in Y$. From Step 1, $\{u^{j_{i}}\}_{i}$ converges weakly to $u$,
	so $u=\tilde{u}\in Y$. Now we prove the claim. Let $u\in Y$. By Lemma \ref%
	{Heat-kernel3} and Theorem \ref{Key-theorem1} we obtain%
	\begin{eqnarray*}
		\big\|F(u)(t,\cdot )\big\|_{\dot{K}_{p,q}^{\alpha }F_{\beta }^{s}} &\leq
		&\int_{0}^{t}\big\|e^{(t-\tau )\Delta }(G(u)(\tau ,\cdot ))\big\|_{\dot{K}%
			_{p,q}^{\alpha }F_{\beta }^{s}}d\tau \\
		&\leq &C\int_{0}^{t}(t-\tau )^{-\frac{s-s\mu }{2}}\big\|G(u)(\tau ,\cdot )%
		\big\|_{\dot{K}_{p,q}^{\alpha }F_{\beta }^{s_{\mu }}}d\tau \\
		&\leq &C\int_{0}^{t}(t-\tau )^{-\frac{s-s\mu }{2}}\big\|u\big\|_{\dot{K}%
			_{p,q}^{\alpha }F_{\beta }^{s}}^{\mu }d\tau \\
		&\leq &CT^{1-\frac{s-s\mu }{2}}\big\|u\big\|_{Y}^{\mu }.
	\end{eqnarray*}%
	This leads to \eqref{uj}, with the help of the fact that%
	\begin{equation*}
	\big\|u^{0}\big\|_{\dot{K}_{p,q}^{\alpha }F_{\beta }^{s}}\leq C\big\|u_{0}%
	\big\|_{\dot{K}_{p,q}^{\alpha }F_{\beta }^{s}}
	\end{equation*}%
	by, Lemma \ref{Heat-kernel3} and 
	\begin{equation*}
	1-\frac{s-s\mu }{2}=\frac{(\mu -1)(s-\bar{s})}{2}>0.
	\end{equation*}%
	From \eqref{T-hyp}, we easily obtain $T_{0}\geq C\big\|u_{0}\big\|_{\dot{K}%
		_{p,q}^{\alpha }F_{\beta }^{s}}^{\frac{-2}{s-\bar{s}}}$.
	
	\textbf{Substep 1.3.}\textit{\ }We will prove the uniqueness of the solution
	of \eqref{int-equa}. Let $u,v\in Y$ be two solutions for the same initial
	data $u_{0}$. Using the fact that $u$ and $u$ solve \eqref{int-equa}, we
	obtain 
	\begin{equation*}
	\big\|u-v\big\|_{X}=\big\|F(u)-F(v)\big\|_{X}\leq 2CT^{\frac{(\mu -1)(s-\bar{%
				s})}{2}}A^{\mu -1}\big\|u-v\big\|_{X}
	\end{equation*}%
	where%
	\begin{equation*}
	A=\sup_{t\in \lbrack 0,T]}(\big\|u(t\cdot )\big\|_{\dot{K}_{\tilde{p}%
			,q}^{0}}^{\mu -1},\big\|v(t\cdot )\big\|_{\dot{K}_{\tilde{p},q}^{0}}^{\mu
		-1}),\quad T<\max (T_{0}(u),T_{0}(v)).
	\end{equation*}%
	Taking $T$ small enough such that%
	\begin{equation*}
	2CT^{\frac{(\mu -1)(s-\bar{s})}{2}}A^{\mu -1}<\frac{1}{2}
	\end{equation*}%
	we obtain $u=v$ on $[0,T]$. We iterate this to prove that $T_{0}(u)=T_{0}(v)$
	and $u=v$ on $[0,T_{0}(u))$, which ensures the uniqueness of the solution of %
	\eqref{int-equa}.
	
	\textbf{Step 2.} We prove part (ii) of the theorem. We split our
	considerations into the cases $\theta <2\vartheta (\mu -1)$ and $\theta
	=2\vartheta (\mu -1)$.
	
	\noindent $\bullet $ \textit{Case 1.} $\theta <2\vartheta (\mu -1)$. Let $%
	u\in Y$ be a solution of \eqref{int-equa} with initial data $u_{0}$. Observe
	that $2\vartheta (\mu -1)=(\mu -1)(s-\bar{s})=2-s+s_{\mu }$. Thanks to Lemma %
	\ref{Heat-kernel3} and Theorem \ref{Key-theorem1} it follows 
	\begin{eqnarray*}
		\big\|u-e^{t\Delta }u_{0}\big\|_{\dot{K}_{p,q}^{\alpha }F_{\beta }^{s+\theta
		}} &\leq &\int_{0}^{t}\big\|e^{(t-\tau )\Delta }(G(u)(\tau ,\cdot ))\big\|_{%
			\dot{K}_{p,q}^{\alpha }F_{\beta }^{s+\theta }}d\tau \\
		&\leq &C\int_{0}^{t}(t-\tau )^{-\frac{\theta }{2}-\frac{s-s_{\mu }}{2}}\big\|%
		G(u)(\tau ,\cdot )\big\|_{\dot{K}_{p,q}^{\alpha }F_{\beta }^{s_{\mu }}}d\tau
		\\
		&\leq &CT_{0}^{1-\frac{\theta }{2}-\frac{s-s\mu }{2}}\big\|u\big\|_{Y}^{\mu
		},
	\end{eqnarray*}%
	since $1-\frac{\theta }{2}-\frac{s-s_{\mu }}{2}=-\frac{\theta }{2}+\frac{%
		(\mu -1)(s-\bar{s})}{2}>0$.
	
	\noindent $\bullet $ \textit{Case 2.} $\theta =2\vartheta (\mu -1)$. Observe
	that $s_{\mu }<\mu $, this gives%
	\begin{equation*}
	\mu >\frac{\frac{n}{p}+\alpha }{\frac{n}{p}+\alpha -s+1}=s_{0}>1\quad \text{%
		and}\quad s=1+\frac{s_{0}-1}{s_{0}}\Big(\frac{n}{p}+\alpha \Big).
	\end{equation*}%
	In addition%
	\begin{equation*}
	0<s_{\mu }<s_{0}<\frac{n}{p}+\alpha
	\quad \text{and}\quad s_{0}\geq \frac{\frac{n}{p}+\alpha }{\frac{n}{q}%
		+\alpha +1}.
	\end{equation*}%
	Assume that $s+\theta =2+s_{\mu }<s_{0}$. Let $2+s_{\mu }<s_{1}<s_{0}$ and $%
	0<\gamma <1$ be such that $s+\theta =\gamma s_{\mu }+(1-\gamma )s_{1}$. From
	interpolation inequality \eqref{Interpolation}, Lemma \ref{Heat-kernel3},
	Theorems \ref{Key-theorem1} and \ref{Key-theorem2} we get%
	\begin{equation*}
	\big\|u-e^{t\Delta }u_{0}\big\|_{\dot{K}_{p,q}^{\alpha }F_{\beta }^{s+\theta
	}}
	\end{equation*}%
	can be estimated by 
	\begin{eqnarray*}
		&&\int_{0}^{t}\big\|e^{(t-\tau )\Delta }(G(u)(\tau ,\cdot ))\big\|_{\dot{K}%
			_{p,q}^{\alpha }F_{\beta }^{s+\theta }}d\tau \\
		&\leq &C\int_{0}^{t}\big\|e^{(t-\tau )\Delta }G(u)(\tau ,\cdot )\big\|_{\dot{%
				K}_{p,q}^{\alpha }F_{\beta }^{s_{\mu }}}^{\gamma }\big\|e^{(t-\tau )\Delta
		}G(u)(\tau ,\cdot )\big\|_{\dot{K}_{p,q}^{\alpha }F_{\beta
			}^{s_{1}}}^{1-\gamma }d\tau \\
		&\leq &C\int_{0}^{t}\big\|G(u)(\tau ,\cdot )\big\|_{\dot{K}_{p,q}^{\alpha
			}F_{\gamma \beta }^{s_{\mu }}}^{\gamma }\big\|G(u)(\tau ,\cdot )\big\|_{\dot{%
				K}_{p,q}^{\alpha }F_{\infty }^{s_{0}}}^{1-\gamma }d\tau \\
		&\leq &CT_{0}\big\|u\big\|_{Y}^{\gamma \mu }\big\|u\big\|_{Y}^{(1-\gamma
			)s_{0}}.
	\end{eqnarray*}%
	Now assume that $\theta +s=2+s_{\mu }\geq s_{0}$. Let $\kappa >0$ be such
	that $2+s_{\mu }-s_{0}<\kappa <2$. Let $%
	0<\varrho <1$ be such that $s+\theta =\varrho s_{\mu }+(1-\varrho
	)(s_{0}+\kappa )$. Again, from interpolation inequality we obtain 
	\begin{equation*}
	\big\|u-e^{t\Delta }u_{0}\big\|_{\dot{K}_{p,q}^{\alpha }F_{\beta }^{s+\theta
	}}
	\end{equation*}%
	is bounded by%
	
	\begin{eqnarray*}
		&&\int_{0}^{t}\big\|e^{(t-\tau )\Delta }(G(u)(\tau ,\cdot ))\big\|_{\dot{K}%
			_{p,q}^{\alpha }F_{\beta }^{s+\theta }}d\tau  \\
		&\leq &C\int_{0}^{t}\big\|e^{(t-\tau )\Delta }G(u)(\tau ,\cdot )\big\|_{\dot{%
				K}_{p,q}^{\alpha }F_{\beta \varrho }^{s_{\mu }}}^{\varrho }\big\|e^{(t-\tau
			)\Delta }G(u)(\tau ,\cdot )\big\|_{\dot{K}_{p,q}^{\alpha }F_{\infty
			}^{s_{0}+\kappa }}^{1-\varrho }d\tau .
	\end{eqnarray*}%
	Applying H\"{o}lder's inequality, Lemma \ref{Heat-kernel3}, Theorems \ref%
	{Key-theorem1} and \ref{Key-theorem2}, we estimate the last expression by%
	\begin{eqnarray*}
		&&C\Big(\int_{0}^{t}\big\|e^{(t-\tau )\Delta }G(u)(\tau ,\cdot )\big\|_{\dot{%
				K}_{p,q}^{\alpha }F_{\beta \varrho }^{s_{\mu }}}d\tau \Big)^{\varrho } \\
		&&\times \Big(\int_{0}^{t}\big\|e^{(t-\tau )\Delta }G(u)(\tau ,\cdot )\big\|%
		_{\dot{K}_{p,q}^{\alpha }F_{\infty }^{s_{0}+\kappa }}d\tau \Big)^{1-\varrho }
		\\
		&\leq &C\Big(\int_{0}^{t}\big\|G(u)(\tau ,\cdot )\big\|_{\dot{K}%
			_{p,q}^{\alpha }F_{\beta \varrho }^{s_{\mu }}}d\tau \Big)^{\varrho } \\
		&&\times \Big(\int_{0}^{t}(t-\tau )^{-\frac{\kappa }{2}}\big\|G(u)(\tau
		,\cdot )\big\|_{\dot{K}_{p,q}^{\alpha }F_{\infty }^{s_{0}}}d\tau \Big)%
		^{1-\varrho } \\
		&\leq &CT_{0}^{1+(\varrho -1)\frac{\kappa }{2}}\big\|u\big\|_{Y}^{\varrho
			\mu }\big\|u\big\|_{Y}^{(1-\varrho )s_{0}}.
	\end{eqnarray*}%
	The proof is completed.
\end{proof}

Using a combination of the arguments used in the proof of Theorem \ref%
{Theorem1} with the help of Theorem \ref{Key-theorem2} we get the following
result:

\begin{thm}
	Let $0<p,q<\infty ,0 \leq \alpha<n-\frac{n}{p} ,\mu \geq \frac{\frac{n}{p}+\alpha }{\frac{n%
		}{q}+\alpha +1}$\ and 
	\begin{equation*}
	1<\mu <\frac{n}{p}+\alpha .
	\end{equation*}%
	Let $G\in Lip\mu $ and%
	\begin{equation*}
	s=1+\frac{\mu -1}{\mu }\big(\frac{n}{p}+\alpha \big).
	\end{equation*}%
	$\mathrm{(i)}$ For all initial data $u_{0}$ in $\dot{K}_{p,q}^{\alpha
	}F_{\beta }^{s}$ with $s>\bar{s}$, there exists a maximal solution $u$ to %
	\eqref{int-equa} in $C([0,T_{0}),\dot{K}_{p,q}^{\alpha }F_{\beta }^{s})$
	with $T_{0}\geq C\big\|u_{0}\big\|_{\dot{K}_{p,q}^{\alpha }F_{\beta }^{s}}^{-%
		\frac{1}{\vartheta }}$.$\newline
	\mathrm{(ii)}$ Let $\theta \leq 2\vartheta (\mu -1)$. We have 
	\begin{equation*}
	u-e^{t\Delta }u_{0}\in C([0,T_{0}),\dot{K}_{p,q}^{\alpha }F_{\beta
	}^{s+\theta }).
	\end{equation*}
\end{thm}

Let $s>\frac{n}{p}+\alpha $. Using Theorem \ref{Key-theorem1 copy(1)}, the
embedding $\dot{K}_{p,q}^{\alpha }F_{\beta }^{s}\hookrightarrow L^{\infty }$%
, we immediately arrive at the following result. We omit the proof since is
essentially similar to the proof of Theorem \ref{Theorem1}.

\begin{thm}
	\label{Theorem2}Let $1<p,q<\infty ,1<\beta <\infty ,\mu >1$ and $0 \leq \alpha<n-\frac{n}{p}$. Let $G\in Lip\mu $ and 
	\begin{equation*}
	\frac{n}{p}+\alpha<s<\mu
	.
	\end{equation*}%
	$\mathrm{(i)}$ For all initial data $u_{0}$ in $\dot{K}_{p,q}^{\alpha
	}F_{\beta }^{s}$ with $s>\bar{s}$, there exists a maximal solution $u$ to %
	\eqref{int-equa} in $C([0,T_{0}),\dot{K}_{p,q}^{\alpha }F_{\beta }^{s})$
	with $T_{0}\geq C\big\|u_{0}\big\|_{\dot{K}_{p,q}^{\alpha }F_{\beta }^{s}}^{-%
		\frac{1}{\vartheta }}$.$\newline
	\mathrm{(ii)}$ Let $\theta <2$. We have 
	\begin{equation*}
	u-e^{t\Delta }u_{0}\in C([0,T_{0}),\dot{K}_{p,q}^{\alpha }F_{\beta
	}^{s+\theta }).
	\end{equation*}
\end{thm}

\begin{rem}
	Corresponding statements to Theorem \ref{Theorem1} were proved by Ribaud 
	\cite{Ri98}, with $\theta <2\vartheta (\mu -1),\alpha =0,p=q$ and $\beta =2$%
	, under the assumption 
	\begin{equation}
	\frac{n}{p}-\frac{n}{\mu p}<s<\min \Big(\frac{(1+\frac{n}{p})(\mu -1)}{\mu },%
	\frac{n}{p}\Big).  \label{Ribaud1}
	\end{equation}%
	Here we are requiring%
	\begin{equation*}
	\max \Big(0,\frac{n}{p}-\frac{n}{\mu }\Big)<s<\min \Big(1+\frac{\mu -1}{\mu }%
	\frac{n}{p},\frac{n}{p}\Big),
	\end{equation*}%
	which improve \eqref{Ribaud1}.
\end{rem}

\textbf{Acknowledgements. }   This work is founded by the General Direction of Higher Education and Training under\ Grant No. C00L03UN280120220004 and by
The General Directorate of Scientific Research and Technological Development.

\bigskip

\bigskip

Douadi \ Drihem

M'sila University

Department of Mathematics

Laboratory of Functional Analysis and Geometry of Spaces

P.O. Box 166, M'sila 28000, Algeria

E-mail: douadidr@yahoo.fr, douadi.drihem@univ-msila.dz

\end{document}